\documentclass{amsart}
\usepackage[T1]{fontenc}
\usepackage{graphicx, amsmath, amsfonts, amsthm, bbding, stix, bbm, hyperref, apptools}
\usepackage[backend=biber]{biblatex}
\usepackage{enumerate}
\usepackage{lipsum}
\usepackage{fancyhdr}
\newtheorem{theorem}[subsection]{Theorem}
\newtheorem{remark}[subsection]{Remark}
\usepackage{listofitems}
\newcommand\cycle[2][\,]{%
  \readlist\thecycle{#2}%
  (\foreachitem\i\in\thecycle{\ifnum\icnt=1\else#1\fi\i})%
}

=650pt \headheight = 8pt \headsep = 10pt

\title{A Sharp Higher Order Sobolev Inquality on Riemannian Manifolds}
\author{samuel zeitler}
\date{September 2024}
\thanks{This work is part of the author's thesis under the supervision of J. Vétois, submitted to McGill University on May 31, 2024.}
\newcommand{\bbR}{\ensuremath{\mathbb{R}}}
\newtheorem{lem}[subsection]{Lemma}
\AtAppendix{\counterwithin{subsection}{section}}

\begin{document}
\begin{abstract}
Let \( m, n \) be integers such that \( \frac{n}{2} > m \geq 1 \) and let \( (M, g) \) be a closed \( n-\)dimensional Riemannian manifold. We prove there exists some \( B \in \mathbb{R} \) depending only on \( (M, g) \), \( m \), and \( n \) such that for all \( u \in H_m^2(M) \),
\[ \lVert u \rVert_{2^\#}^2 \leq K(m,n) \int_M (\Delta^\frac{m}{2} u)^2 dv_g + B \lVert u \rVert_{H_{m-1}^2(M)}^2 \] 
where \( 2^\# = \frac{2n}{n-2m} \), \( K(m,n) \) is the square of the best constant for the embedding \( W^{m,2}(\bbR^n) \subset  L^{2^\#}(\bbR^n) \), \( H_m^2(M) \) is the Sobolev space consisting of functions on \( M \) with \( m \) weak derivatives in \( L^2(M) \), and \( \Delta^\frac{m}{2} = \nabla \Delta^{\frac{m-1}{2}} \) if \( m \) is odd. This inequality is sharp in the sense that \( K(m,n) \) cannot be lowered to any smaller constant. This extends the work of Hebey-Vaugon\cite{hebeyvaugon} and Hebey\cite{hebeysecond} which correspond respectively to the cases \( m=1 \) and \( m=2 \).
\end{abstract}
\maketitle

\section{Introduction}

In part due to applications to PDEs, geometry, and physics, sharp constants in Sobolev inequalities have been the subject of significant attention for many decades. The best constants for the first order Sobolev embeddings \( W^{1,p}(\bbR^n) \subset L^{p^*}(\bbR^n) \) where \( 1 < p < n \) and \( p^* = \frac{np}{n-p} \) were computed independently by Rodemich\cite{rodemich},  Aubin\cite{aubinsob}, and Talenti\cite{talentisob}. In the case \( p=2 \), this result was generalized to second order embeddings in Edmunds-Fortunato-Jannelli\cite{edmunds} and to all higher order by Swanson\cite{swanson}. In particular, they showed that if \( \frac{n}{2} > m \geq 1 \), \( 2^\# = \frac{2n}{n-2m} \), and
\[ K(m,n)^{-1} =  \inf_{u \in W^{m,2}(\bbR^n) \setminus \{0\}} \frac{\lVert \Delta^\frac{m}{2} u \rVert_{2}^2 }{ \lVert u \rVert_{2^\#}^2 }, \] then this value is exactly
\[ K(m,n) = \pi^{-m}\left( \frac{\Gamma(n)}{\Gamma(n/2)} \right)^\frac{2m}{n} \prod_{l=-m}^{m-1}(n+2l)^{-1}. \]

Assume \( \frac{n}{2} > m \geq 1 \) and let \( (M, g) \) be a closed \( n-\)dimensional Riemannian manifold. The following asymptotically sharp Sobolev inequality was proven for \( m=1 \) by Aubin\cite{aubinsob}, \( m=2 \) by Djadli-Hebey-Ledoux\cite{djadlihebeyledoux}, and \( m \geq 3 \) by Mazumdar\cite{mazumdargjms}: For all \( \epsilon > 0 \) there exists \( B_\epsilon \in \bbR \) such that for all \( u \in H_m^2(M) \),
\begin{equation}\label{asympsharpfirst} \lVert u \rVert_{2^{\#}}^2 \leq (K(m,n) +\epsilon) \int_M (\Delta^\frac{m}{2} u)^2 dv_g + B_\epsilon\lVert u \rVert_{H_{m-1}^2(M)}^2. \end{equation}
It was additionally shown that \( K(m,n) \) is the best possible value of \( A \) for which the inequality 
\[ \lVert u \rVert_{2^{\#}}^2 \leq A \int_M (\Delta^\frac{m}{2} u)^2 dv_g + B\lVert u \rVert_{H_{m-1}^2(M)}^2. \]
could hold (regardless of the value of \( B \)). In the case \( m=1 \), Aubin\cite{aubinsob} conjectured this best constant is attained (i.e. one can take \( B_\epsilon \) in \( \ref{asympsharpfirst} \) to be bounded as \( \epsilon \to 0 \)). This conjecture was solved by Hebey and Vaugon\cite{hebeyvaugon}. We also mention if additionally \( n \geq 6 \), a refinement of their inequality involving the scalar curvature was proven by Li and Ricciardi\cite{liricciardi}. The result of Hebey-Vaugon\cite{hebeyvaugon} was extended to \( m=2 \) in Hebey\cite{hebeysecond}. The following theorem extends this result to all \( m \geq 3 \).
\begin{theorem}
    Let \( m \geq 1 \) and \( (M, g) \) be a closed \( n-\)dimensional Riemannian manifold such that \( n > 2m \). Then there exists some \( B \in \mathbb{R} \) depending only on \( (M, g) \), \( m \), and \( n \) such that for all \( u \in H_m^2(M) \),
\begin{equation}\label{abhighersimp} \lVert u \rVert_{2^\#}^2 \leq K(m,n) \int_M (\Delta^\frac{m}{2} u)^2 dv_g + B \lVert u \rVert_{H_{m-1}^2(M)}^2. \end{equation}
\end{theorem}

\begin{remark}
    Independently from our work, Theorem 1.1 has also been obtained in a recent article by Carletti\cite{carletti1}. The method of Carletti relies on a delicate a priori asymptotic description of solutions of equation (4) below, as well as on pointwise estimates of Green's functions for the operator \( (\Delta + \alpha)^m \) obtained by the same author in \cite{carletti2}. We believe both methods of proof have their independent advantages and merits.
\end{remark}

Section 2 is devoted to the proof of Theorem 1. Our proof adapts the method of Hebey\cite{hebeysecond} in the second order case. The argument is by contradiction and proceeds with a concentration point analysis of a sequence of smooth positive solutions \( u_\alpha \) to the PDE \( (\Delta_g + \alpha)^m u = \lambda_\alpha u^{2^{\#} - 1} \) for \( \alpha > 0 \) and \( 0 < \lambda_\alpha \to \frac{1}{K(m,n)} \) as \( \alpha \to \infty \). This general method for proving sharp Sobolev inequalities is by now classical, other examples can be found in Hebey-Vaugon\cite{hebeyvaugon}, Druet\cite{druetbest}, Aubin-Li\cite{aubinli}, Li-Ricciardi\cite{liricciardi}, and Li-Zhu\cite{lizhu}\cite{lizhu2}. We show the functions \( u_\alpha \) concentrate at one unique point \( x_0 \) up to a subsequence and decrease sufficiently fast away from their maximum points \( x_\alpha \) (see Step 2 in Section 2.1) to obtain a contradiction to the sharp Euclidean Sobolev inequality.

We note some arguments in the proof which are unique to the higher order case. In Step 1, we require \( L^2 \) bounds on functions \( \tilde{v}_\alpha^{(k)} \) defined in \( (\ref{vdef}) \) for \( 1 \leq k \leq m-1 \). In order to obtain these estimates, we use separate arguments for the cases \( k \leq \lfloor \frac{m}{2} \rfloor \) and \( k > \lfloor \frac{m}{2} \rfloor \), where if \( m=2 \) only the first case is relevant. In Step 3, an additional argument is needed to obtain the estimate \[ \int_{M \setminus B_{x_0}( \delta)} (\Delta^\frac{k}{2} u_\alpha)^2 dv_g \leq C \alpha^k \int_{M \setminus B_{x_0}( 2^{-(2k+1)}\delta)} u^2 dv_g \]  for \( 1 \leq k \leq m-1 \), compared to the case \( m=2 \) where this is equivalent to \[ \int_{M \setminus B_{x_0}( \delta)} |\nabla u_\alpha|^2 dv_g \leq C\alpha^{m-1} \int_{M \setminus B_{x_0}( \delta/8)} dv_g. \]

One of the difficulties in adapting the argument for \( m=1 \) and \( m=2 \) to prove Theorem 1.1 is the rapidly increasing number of terms present in the calculations as \( m \) increases, which must be managed by careful inductions. To address this technical aspect of the problem, Appendix A contains several lemmas which are used frequently throughout the proof in Section 2. The majority of the lemmas are centered around the ability to integrate by parts expressions such as
\( \int_M \eta \Delta^{i_1} u \Delta^{i_2} u dv_g \)
and expand Laplacians of products in expressions such as
\( \int_M (\Delta^i (\eta u))^2 dv_g \) for \( u \in H_m^2(M) \), where \( \eta \) is some fixed smooth function.

\section{Proof of Theorem 1.1}

In this section, all unlabeled Laplacians are assumed to be with respect to the Riemannian metric \( g \) and all integrals without volume elements are assumed to be with respect to the Riemannian volume element \( dv_g \). We occasionally choose to label them anyway, in particular during portions of the argument when other metrics are involved in order to reduce confusion. To ease notation we refer to the constant \( K(m,n) \) as \( K \).

Suppose for all \( B > 0 \), inequality \( (\ref{abhighersimp}) \) is false. Then if we let \( \mathcal{N} \coloneq \{ u \in H_m^2(M) : \int_M |u|^{2^\#} = 1 \} \), for all \( \alpha > 0 \) we have
  \begin{equation}\label{infless} \lambda_\alpha \coloneq \inf_{u \in \mathcal{N} } \int_M u (\Delta + \alpha)^m u   < \frac{1}{K}.  \end{equation}
where we interpret the integral in the sense of distributions and use the negative Laplacian convention \( \Delta = -\nabla^i \nabla_i \). The sequence \( \lambda_\alpha \) is clearly nondecreasing and so converges to some \( \lambda \leq \frac{1}{K} \). Then by Mazumdar\cite{mazumdargjms}, for all \( \alpha > 0 \) there exists a \( C^{2m}(M) \) solution  \( u_\alpha \in \mathcal{N} \) to the PDE
\begin{equation*}
    (\Delta + \alpha)^m u_\alpha = \lambda_\alpha |u_\alpha|^{2^\#-2}u_\alpha. \end{equation*}
By making trivial adaptations to Proposition 4.1 in Robert\cite{robertfourthorder}, \( u_\alpha \) can be taken to be smooth and positive and therefore solves
  \begin{equation}\label{originalpde} (\Delta + \alpha)^m u_\alpha = \lambda_\alpha u_\alpha^{2^\#-1}. \end{equation}
  Independently, by Mazumdar\cite{mazumdargjms}, for all \( \epsilon > 0 \) there exists \( B_\epsilon \) such that for all \( u \in H_m^2(M) \), 
 \begin{equation}\label{asympsharp} \left( \int_M |u|^{2^\#} \right)^\frac{2}{2^\#}  \leq (K+\epsilon) \int_M (\Delta^\frac{m}{2} u)^2  + B_\epsilon\lVert u \rVert_{H_{m-1}^2}. \end{equation}
 Therefore, for all \( \epsilon \) there exists \( \alpha_\epsilon \) large such that
 \[ \frac{1}{(1+\epsilon)K} \leq \inf_{u \in \mathcal{N} } \int_M u (\Delta + \alpha_\epsilon)^m u . \]
 Combining this with \( (\ref{infless}) \) shows \( \lambda = \frac{1}{K} \).
We note that \( \lVert u_\alpha \rVert_{H_{m-1}^2} \to 0 \) as \( \alpha \to \infty \) while \( \lVert u_\alpha \rVert_{2^{\#}} = 1 \). Then because
 \[ \int_M u^{2^\#}_\alpha dv_g \leq (\sup_M u_\alpha)^{2^\#-2} \int_M u_\alpha^2 \]
 we must have \( \sup_M u_\alpha \to \infty \) as \( \alpha \to \infty \). Let \( x_\alpha \in M \) be such that \( u_\alpha(x_\alpha) \) is maximum. Taking a subsequence of \( \alpha \) we assume \( x_\alpha \to x_0 \in M \) and \( u_\alpha(x_\alpha) \) increases to infinity.

 \subsection{Asymptotic Analysis}

  We say a point \( y \in M \) is a concentration point for the sequence \( u_\alpha \) if for all \( \delta > 0 \)
 \[ \limsup_{\alpha \to \infty} \int_{B_y(\delta)} u^{2^{\#}}_\alpha > 0 \]

 \textbf{Step 1:} Up to a subsequence, \( u_\alpha \) has one unique concentration point.

 Let us define \( \mu_\alpha = u_\alpha(x_\alpha)^{-\frac{2}{n-2m}}\). Let \( i_g \) be the injectivity radius for \( (M, g) \). We define sequences of functions \( \tilde{u}_\alpha\) and metrics \( \tilde{g}_\alpha \) on \( B_0(i_g/\mu_\alpha) \) by

 \[ \tilde{u}_\alpha(x) = \mu_\alpha^{\frac{n-2m}{2}} u_\alpha( \text{exp}_{x_\alpha}(\mu_\alpha x) ) \text{ and } \tilde{g}_\alpha(x) = (\text{exp}^*_{x_\alpha}g)(\mu_\alpha x ). \]
We note \( \bigcup_\alpha B_0(i_g/\mu_\alpha) = \mathbb{R}^n \).
The functions \( \tilde{u}_\alpha \) are bounded in \( C^0(\mathbb{R}^n) \) with \( \sup \tilde{u}_\alpha = \tilde{u}(0) = 1 \) and they satisfy the PDE

 \[ (\Delta_{\tilde{g}_\alpha} + \alpha \mu_\alpha^2)^m \tilde{u}_\alpha = \lambda_\alpha \tilde{u}_\alpha^{2^{\#}-1} .\]

We would like to show these functions are uniformly bounded in \( C^{2m,\beta}(K) \) for \( K \) compact and therefore converge to a limit function \( \tilde{u} \). To show boundedness in \( C^{2m,\beta}(K) \), it suffices to show \( \alpha \mu_\alpha^2 \) is bounded as \( \alpha \to \infty \) and apply standard regularity theory as seen in Gilbarg-Trudinger\cite{gt}. Suppose up to a subsequence \( \alpha \mu_\alpha^2 \to \infty \) as \( \alpha \to \infty \). 
Let us define for \( 0 \leq k \leq m \) functions \( \tilde{v}_\alpha^{(k)} \) and \( v_\alpha^{(k)} \) on \( B_0(i_g/\mu_\alpha) \) and \( M \) respectively by
\begin{equation}\label{vdef} 
\tilde{v}_\alpha^{(k)} = (\Delta_{\tilde{g}_\alpha} + \alpha \mu_\alpha^2)^k \tilde{u}_\alpha = \sum_{i=0}^k c_{i,k} \alpha^{k-i}\mu_\alpha^{2(k-i)} \Delta_{\tilde{g}_\alpha}^i \tilde{u}_\alpha \end{equation}
and
\[
v_\alpha^{(k)} = (\Delta_{{g}} + \alpha)^k {u}_\alpha = \sum_{i=0}^k c_{i,k} \alpha^{k-i} \Delta_g^i u_\alpha
\] where \( c_{i,k} \) are defined the same as in Lemma A.7.

We obtain systems of PDEs
\begin{equation}\label{systemtilde} \begin{cases}
    \Delta_{\tilde{g}_\alpha} \tilde{v}^{(k)}_\alpha + \alpha \mu_\alpha^2 \tilde{v}^{(k)}_\alpha = \tilde{v}^{(k+1)}_\alpha & \text{if $0 \leq k\leq m-2$}  \\
    \Delta_{\tilde{g}_\alpha} \tilde{v}^{(m-1)}_\alpha + \alpha \mu_\alpha^2\tilde{v}^{(m-1)}_\alpha = \tilde{u}^{2^{\#} - 1}_\alpha 
\end{cases} \end{equation}
and
\begin{equation}\begin{cases}\label{systemwithouttilde}
    \Delta_g v^{(k)}_\alpha + \alpha v^{(k)}_\alpha = v^{(k+1)}_\alpha & \text{if $0 \leq k\leq m-2$}\\
    \Delta_g v^{(m-1)}_\alpha + \alpha v^{(m-1)}_\alpha = u^{2^{\#} - 1}_\alpha .
\end{cases} \end{equation}
Because \( u^{2^{\#} - 1}_\alpha \geq 0 \), iterating the maximum principle we obtain \( v^{(k)}_\alpha \geq 0 \) for all \( 1 \leq k \leq m \). This then shows for \( x \in B_0(i_g/\mu_\alpha) \) \[ \tilde{v}^{(k)}_\alpha(x) = \mu_\alpha^{\frac{n}{2}-(m-2k)} v^{(k)}_\alpha(exp_{x_\alpha} (\mu_\alpha x)) \geq 0 \] for all \( 1 \leq k \leq m \).

We now show \( \tilde{v}^{(k)}_\alpha \) is bounded in \( L^2(B_0(R)) \) for fixed \( R > 0 \) as \( \alpha \to \infty \) for any \( 1 \leq k \leq m-1 \). We perform the proof of this claim through the following (strong) induction:
\begin{enumerate}[(i)]
    \item If \( k \leq \lfloor \frac{m}{2} \rfloor \) then \( \int_{B_0(R)} (v_\alpha^{(k)})^2 dv_{\tilde{g}_\alpha} \leq C \).
    \item If \( \lfloor \frac{m}{2} \rfloor \leq k \leq m-2 \) and \( \int_{B_0(R)} (v_\alpha^{(i)})^2 dv_{\tilde{g}_\alpha} \leq C \) for \( i \leq k \), then \( \int_{B_0(R-2)} (v_\alpha^{(k+1)})^2 dv_{\tilde{g}_\alpha} \leq C\).
\end{enumerate}

Let \( k \leq \lfloor \frac{m}{2} \rfloor \). 
Integrating by parts gives
\[ \int_M (v^{(k)}_\alpha)^2 dv_g = \int_M \sum_{i=0}^{2k} d_{i,k} \alpha^{2k-i} u_\alpha \Delta_g^i u_\alpha dv_g \leq C \int_M \sum_{k=0}^{2k} \alpha^{2k-i} u_\alpha \Delta_g^i u_\alpha dv_g \]
where \( d_{i,k} = \sum_{j_1 + j_2 = i} c_{j_1,k}c_{j_2, k} \) and we use the fact that for all integers \( k \geq 0 \) and \( u \in H_{2k}^2 \),
\[ \int_M u \Delta_g^k u dv_g = \int_M (\Delta_g^\frac{k}{2}u)^2 dv_g \geq 0.\] 
Therefore, using the PDE we obtain
\[  \alpha^{m-2k} \int_M (v^{(k)}_\alpha)^2 dv_g\leq C \int_M \sum_{k=0}^{2k} \alpha^{m-i} u_\alpha(\Delta_g^i u_\alpha)dv_g \leq C \int_M u_\alpha (\Delta_g + \alpha)^m u_\alpha dv_g = C \int_M u_\alpha^{2^\#} dv_g \leq C  \]
We therefore have by a change of variable
\begin{align*} \int_{B_0(R)} (\tilde{v}_\alpha^{(k)})^2 dv_{\tilde{g}_\alpha} &= \int_{B_0(R)} \mu_\alpha^{n-2(m-2k)} (v^{(k)}_\alpha(exp_{x_\alpha} (\mu_\alpha x)))^2 dv_{\tilde{g}_\alpha} \\
&= \int_{B_{x_\alpha}(R\mu_\alpha)} \mu^{-2(m-2k)}_\alpha (v^{(k)}_\alpha)^2 dv_g \\
&\leq C \mu^{-2(m-2k)}_\alpha \alpha^{-(m-2k)}\\
& \leq C  \end{align*}
and we have proven (i).

Now let us define for \( 1 \leq i \leq 2(m-1) \) functions \( w^{(i)}_\alpha \) by
\[ \begin{cases}
    w^{(i)}_\alpha = \left(\tilde{v}^{\left(\frac{i}{2}\right)}_\alpha\right)^2 & \text{ if }  i \text{ is even} \\
     w^{(i)}_\alpha = \tilde{v}^{\left(\frac{i}{2} - \frac{1}{2}\right)}_\alpha \tilde{v}^{\left(\frac{i}{2} + \frac{1}{2} \right)}_\alpha & \text{ if }  i \text{ is odd}. \\
\end{cases} \]

Let \( \lfloor \frac{m}{2} \rfloor \leq k \leq m-2 \) and let \( \eta \) be a smooth nonnegative function such that \( \eta = 1 \) on \( B_0(R-1) \) and \( \eta = 0 \) on \( \mathbb{R}^n \setminus B_0(R- \frac{1}{2}) \). Assume \( \int_{B_0(R)} (v_\alpha^{(l)})^2 dv_{\tilde{g}_\alpha} \leq C \) for \( l \leq k \). This is equivalent to \( \int_{B_0(R)} w_\alpha^{(i)} dv_{\tilde{g}_\alpha} \leq C \) for \( i \leq 2k \). Because \( k \geq \lfloor \frac{m}{2} \rfloor \) there exists a nonnegative integer \( s \) such that \( m + s = 2k+1 \). Additionally because \( k \leq m-2 \) we must have \( s+1 \leq k \). Then applying Lemma A.8 with \( r = \frac{1}{2} \) gives

\[ \int_{B_0(R-1)} \tilde{v}^{(k+1)}_\alpha \tilde{v}^{(k)}_\alpha dv_{\tilde{g}_\alpha} \leq \int_{B_0(R - \frac{1}{2})} \eta \tilde{v}^{(k+1)}_\alpha\tilde{v}^{(k)}_\alpha dv_{\tilde{g}_\alpha} \leq \int_{B_0(R- \frac{1}{2})} \eta \tilde{v}^{(m)}_\alpha \tilde{v}^{(s)}_\alpha dv_{\tilde{g}_\alpha} + C \int_{B_0(R)} \sum_{i=0}^{2k} w^{(i)}_\alpha dv_{\tilde{g}_\alpha} . \]
We note because \( \tilde{g}_\alpha  \) converges to the Euclidean metric uniformly on compact subsets (and its derivatives up converge to \( 0 \) uniformly on compact subsets), the constant \( C \) above is indeed independent of \( \alpha \). Substituting definitions of \( w^{(i)}_\alpha \), applying Cauchy-Schwarz, recognizing \( \tilde{v}^{(m)}_\alpha = \tilde{u}^{2^\#-1}_\alpha \leq 1 \), and using the fact that \( s \leq k \) gives
\[ \int_{B_0(R-1)} w^{(2k+1)} dv_{\tilde{g}_\alpha} \leq \left(\int_{B_0(R)} (\tilde{v}^{(m)}_\alpha)^2 dv_{\tilde{g}_\alpha} \right)^\frac{1}{2} \left( \int_{B_0(R)} (w^{(2s)}_\alpha) dv_{\tilde{g}_\alpha} \right)^\frac{1}{2} +  C \int_{B_0(R)} \sum_{i=0}^{2k} w^{(i)}_\alpha dv_{\tilde{g}_\alpha} \leq C. \]
Therefore \( \int_{B_0(R-1)} w^{(i)} dv_{\tilde{g}_\alpha} \leq C \) for \( i \leq 2k+1 \). Now let \( \eta \) be a smooth function such that \( \eta = 1 \) on \( B_0(R-2) \) and \( \eta = 0 \) on \( \mathbb{R}^n \setminus B_0(R- \frac{3}{2}) \). Again applying Lemma A.8 with \( r = \frac{1}{2} \),
\[ \int_{B_0(R-2)} (\tilde{v}^{(k+1)}_\alpha)^2 dv_{\tilde{g}_\alpha} \leq \int_{B_0(R- \frac{3}{2})} \eta (\tilde{v}^{(k+1)}_\alpha)^2 dv_{\tilde{g}_\alpha} \leq \int_{B_0(R- \frac{3}{2})} \eta \tilde{v}^{(m)}_\alpha \tilde{v}^{(s+1)}_\alpha dv_{\tilde{g}_\alpha} + C \int_{B_0(R-1)} \sum_{i=0}^{2k+1} w^{(i)}_\alpha dv_{\tilde{g}_\alpha} . \]
Then we once again apply Cauchy-Schwarz, \(  \tilde{v}^{(m)}_\alpha \leq 1 \) and the fact that \( s+1 \leq k \) to obtain
\[ \int_{B_0(R-2)} (\tilde{v}^{(k+1)}_\alpha)^2 dv_{\tilde{g}_\alpha} \leq C \]
and we have shown \( (ii) \), therefore for all \( 1 \leq k \leq m-1 \), \( \int_{B_0(R)} (\tilde{v}^{(k)}_\alpha)^2 dv_{\tilde{g}_\alpha} \leq C \).

We now consider \( \tilde{v}^{(m-1)}_\alpha \) which satisfies the last PDE in the system \( (\ref{systemtilde}) \).
 Because \( \tilde{v}^{(m-1)}_\alpha \) is positive and bounded in \( L^2(B_0(R)) \) and \( u^{2^\# - 1}_\alpha  \) is bounded in \( C^0(B_0(R)) \), we get \( \tilde{v}^{(m-1)}_\alpha\) is bounded in \( C^0(B_0(R-1)) \) by the De Georgi-Nash-Moser iteration scheme. We note because \( \tilde{g}_\alpha \) converges to the Euclidean metric uniformly on compact sets, the constant \( C \) in the De Georgi-Nash-Moser iterative scheme will indeed be independent of \( \alpha \).
Then for \( 1 \leq k \leq m-2 \), we consider \( v^{(k)}_\alpha \) in the system \( (\ref{systemtilde}) \) and
apply the De Georgi-Nash-Moser iteration scheme again. If \( \tilde{v}^{(k+1)}_\alpha \) is bounded in \( C^0(B_0(R)) \) and \( \tilde{v}^{(k)}_\alpha \) is bounded in \( L^2(B_0(R)) \) then \( \tilde{v}^{(k)}_\alpha \) is bounded in \( C^0(B_0(R-1)) \).
Therefore by induction \( \tilde{v}^{(k)}_\alpha \) is \( C^0 \) bounded on compact sets for \( 1 \leq k \leq m \). We then take the particular equation from \( (\ref{systemtilde}) \)
\[ \Delta_{\tilde{g}_\alpha} \tilde{u}_\alpha + \alpha \mu^2_\alpha \tilde{u}_\alpha = \tilde{v}^{(1)}_\alpha. \] Because \( \Delta_{\tilde{g}_\alpha} \tilde{u}_\alpha(0) \geq 0 \), \( \tilde{u}_\alpha(0) = 1 \), and \( \tilde{v}_\alpha^{(1)}(0) \leq C \) we obtain that \( \alpha \mu^2_\alpha \) is in fact bounded. We then apply standard regularity theory to see \( \tilde{u}_\alpha \) are uniformly bounded in \( C_{2m}^\beta(K) \)  for \( K \) compact, and therefore \( \tilde{u}_\alpha  \) will converge in \( C_{2m}(\mathbb{R}^n) \) to a limit function \( \tilde{u} \) satisfying \( 0 \leq \tilde{u} \leq 1 \) and \( \tilde{u}(0) = 1 \).

Independently, by integrating by parts, substituting our PDE, and using the fact that \( \int_M u_\alpha^{2^\#} dv_g = 1 \) we have

\begin{equation}\label{ibptosobolev} \sum_{i=0}^m \int_M c_{i,m}\alpha^{m-i}(\Delta^{\frac{i}{2}}u_\alpha)^2 dv_g = \int_M u_\alpha (\Delta + \alpha)^m u_\alpha dv_g  \leq \frac{1}{K}\int_M u_\alpha^{2^\#} dv_g = \frac{1}{K}\left(\int_M u_\alpha^{2^\#} dv_g\right)^\frac{2}{2^\#}.  \end{equation}
We note by the inequalities above and the fact that \( c_{m,m} = 1 \) we have 
\begin{equation}\label{boundonderiv} \lVert \Delta^{\frac{m}{2}}  u_\alpha \rVert_2^2 \leq \frac{1}{K} \lVert u_\alpha \rVert_{2^\#}^2 = \frac{1}{K}. \end{equation}
Given \( \epsilon > 0 \), we apply \( (\ref{ibptosobolev}) \), the asymptotically sharp Sobolev inequality \( (\ref{asympsharp})\), and \( (\ref{boundonderiv}) \) to obtain
\[  \sum_{i=0}^m \int_M c_{i,m}\alpha^{m-i}(\Delta^{\frac{i}{2}}u_\alpha)^2 dv_g \leq \lVert \Delta^{\frac{m}{2}} u_\alpha \rVert_2^2 +  \frac{\epsilon}{K} \lVert \Delta^{\frac{m}{2}}  u_\alpha \rVert_2^2 + \frac{B_\epsilon}{K} \lVert u_\alpha \rVert_{H_{m-1}^2} \leq \lVert \Delta^{\frac{m}{2}} u_\alpha \rVert_2^2 +  \frac{\epsilon}{K^2} + \frac{B_\epsilon}{K} \lVert u_\alpha \rVert_{H_{m-1}^2}.  \]Then because \( c_{0,m}=1=c_{m,m} \) we have for sufficiently large \( \alpha \)
\[ \alpha^m \int_M u_\alpha^2 dv_g \leq \alpha^m \int_M u_\alpha^2 dv_g + \sum_{i=1}^{m-1} \int_M (c_{i,m}\alpha^{m-i}-\frac{B_\epsilon}{K})(\Delta^{\frac{i}{2}}u_\alpha)^2 dv_g \leq \frac{\epsilon}{K^2} +  \frac{B_\epsilon}{K} \int_M u^2_\alpha dv_g = \frac{\epsilon}{K^2}+ o(1)  \]
where \( o(1) \to 0 \) as \( \alpha \to \infty \).
Therefore by letting \( \epsilon \to 0 \) we see \[ \alpha^m \int_M u_\alpha^2 dv_g \to 0  \text{ as } \alpha \to \infty.\]

Because \( \tilde{u} \) is continuous and \( \tilde{u}(0) = 1 \), \( \int_{B_0(1)} \tilde{u}^2 dx > 0 \). Thus by the uniform convergence of \( \tilde{u}_\alpha \to \tilde{u} \) as \( \alpha \to \infty \) on compact sets, there exists some \( c > 0 \) independent of \( \alpha \) such that \( \int_{B_0(1)} \tilde{u}_\alpha^2 \geq c \). Therefore given \( \delta > 0 \) small and fixed, for sufficiently large \( \alpha \) by change of variable
\[ \alpha^m \int_{B_{x_\alpha(\delta)}} u_\alpha^2 dv_g = \alpha^m \mu_\alpha^{2m} \int_{B_0(\frac{\delta}{\mu_\alpha})} \tilde{u}_\alpha^2 dv_{g_\alpha} \geq \alpha^m \mu_\alpha^{2m} \int_{B_0(1)} \tilde{u}_\alpha^2 \geq c \alpha^m \mu_\alpha^{2m} \] 
which implies \( \alpha \mu_\alpha^{2} \to 0 \) as \( \alpha \to \infty \). Hence, passing to the limit, \( \tilde{u} \) is a \( C^{2m}(\mathbb{R}^n) \) nonnegative solution to the PDE 
\[ \Delta_\xi^m \tilde{u} = \frac{1}{K} \tilde{u}^{2^\#-1}. \]
in the Euclidean metric.
By the Euclidean sharp Sobolev inequality (as stated in Lions\cite{lions}) we have
\[ \left( \int_{\mathbb{R}^n} \tilde{u}^{2^\#} dx \right)^\frac{2}{2^\#} \leq K \int_{\mathbb{R}^n} (\Delta_\xi^\frac{m}{2} \tilde{u})^2 dx \]
and so we obtain (after integrating by parts)
\begin{equation}\label{uequalsone} \left( \int_{\mathbb{R}^n} \tilde{u}^{2^\#} dx \right)^\frac{2}{2^\#} \leq \int_{\mathbb{R}^n} \tilde{u}^{2^\#} dx .\end{equation}
By a change of variable, for all \( R > 0 \) and sufficiently large \( \alpha \) we have
\[ \int_{B_0(R)} \tilde{u}_\alpha^{2^\#} dv_{g_\alpha} = \int_{B_{x_\alpha}(\mu_\alpha R)} u_\alpha^{2^\#} dv_g \leq 1 \]
and taking the limit as \( \alpha \to \infty \) and then as \( R \to \infty \) implies \[\int_{B_0(\mathbb{R}^n)} \tilde{u}^{2^\#} \leq 1.\] Hence by \( (\ref{uequalsone}) \) we have \[ \int_{B_0(\mathbb{R}^n)} \tilde{u}^{2^\#} dx =1 \]

Fixing \( R > 0 \) large and \( \delta > 0 \) small, we have
\[ 1 \geq \limsup_{\alpha \to \infty} \int_{B_{x_0}(\delta)} u_\alpha^{2^\#} dv_g \geq \lim_{\alpha \to \infty} \int_{B_{x_\alpha}(\mu_\alpha R)} u_\alpha^{2^\#} dv_g = \int_{B_0(R)} \tilde{u}^{2^\#} dx = 1-\epsilon_R \]
where \( \epsilon_R \to 0 \) as \( R \to \infty \).
We thus obtain \[\lim_{\alpha \to \infty} \int_{B_{x_0}(\delta)} u_\alpha^{2^\#} dv_g = 1. \] Since \( \delta \) is arbitrary, if \( y_0 \neq x_0 \) for any \( \delta < d(x_0, y_0) \) we must have
\[ \lim_{\alpha \to \infty} \int_{B_{y_0}(\delta)} u_\alpha^{2^\#} dv_g = 0 \] otherwise we would contradict \( \int_M u_\alpha^{2^\#} dv_g = 1 \). This completes the proof there is only one unique concentration point \( x_0 \) along our subsequence.

\textbf{Step 2:} For all \( x \in M \) we have the inequality \begin{equation}\label{step2} d_g(x_\alpha, x)^{\frac{n-2m}{2}} u_\alpha(x) \leq C. \end{equation} 

Suppose there is some sequence \( y_\alpha \) such that \begin{equation*} \sup_{x \in M} d_g(x_\alpha, x)^{\frac{n-2m}{2}} u_\alpha(x) =  d_g(x_\alpha, y_\alpha)^{\frac{n-2m}{2}} u_\alpha(y_\alpha) \to \infty. \end{equation*}

Similarly to what is done in Step 1, we define \( \hat{\mu}_\alpha = u_\alpha(y_\alpha)^{-\frac{2}{n-2m}}\) and define the rescalings \( \hat{u}_\alpha \) of \( u_\alpha \) on \( \mathbb{R}^n \) with metric \( \hat{g}_\alpha \) as follows.

 \[ \hat{u}_\alpha = \hat{\mu}_\alpha^{\frac{n-2m}{2}} u_\alpha( \text{exp}_{y_\alpha}(\hat{\mu}_\alpha x) ), \, \, \, \hat{g}_\alpha = (\text{exp}_{y_\alpha}^*g)(\hat{\mu}_\alpha x ). \]
 Then \( \hat{u} \) is a solution to the PDE 
 \[  (\Delta_{\hat{g}_\alpha} + \hat{\mu}_\alpha^2\alpha)^m \hat{u}_\alpha = \lambda_\alpha \hat{u}_\alpha^{2^\#-1}. \]
By our definition of \( y_\alpha \) we additionally have
 \begin{equation}\label{uhatbound} \hat{u}_\alpha(x) \leq \left(\frac{d_g( x_\alpha,y_\alpha)}{d_g(x_\alpha, exp_{y_\alpha} (\hat{\mu}x)}\right)^\frac{n-2m}{2} . \end{equation} 
 Fix \( R > 0 \). If \( |x| \leq R \) the triangle inequality implies \[ d_g(x_\alpha, exp_{y_\alpha} (\hat{\mu}_\alpha x)) \geq d_g(x_\alpha, y_\alpha) - R \hat{\mu}_\alpha.  \]
 Because \( d_g(x_\alpha, y_\alpha)^{\frac{n-2m}{2}} u_\alpha(y_\alpha) \to \infty \), we additionally have \[ \frac{\hat{\mu}_\alpha}{d_g(x_\alpha, y_\alpha)} \to 0. \] Therefore by \( (\ref{uhatbound}) \), \[ \hat{u}_\alpha(x) \leq \left(1 - \frac{R \hat{\mu}_\alpha}{d_g(x_\alpha, y_\alpha)} \right)^{-\frac{n-2m}{2}}  \]
which implies \( \hat{u}_\alpha \) are \( C^0 \) bounded on compact sets.

 Now that we have local \( C^0 \) bounds, following the De Georgi-Nash-Moser iteration scheme argument seen in Step \( 1 \) shows for all \( 1 \leq k \leq m \), the functions \( \hat{v}^{(k)} = (\Delta_{\hat{g}_\alpha} + \hat{\mu}_\alpha^2 \alpha)^k (\hat{u}_\alpha) \) are bounded on compact sets. 

Let us first consider the case (after possibly taking a subsequence) \( y_\alpha \to y_0 \neq x_0 \). Then for all \( \epsilon > 0 \), because \( \Delta_{\hat{g}_\alpha} \hat{u}_\alpha \leq \hat{v}^{(1)}_\alpha \leq C  \),
\[  \Delta_{\hat{g}_\alpha} (\hat{u}_\alpha)^{1+\epsilon} = (1+\epsilon)\hat{u}^\epsilon_\alpha \Delta_{\hat{g}_\alpha} \hat{u}_\alpha - \epsilon(1+\epsilon)\hat{u}_\alpha^{\epsilon - 1}|\nabla_{\hat{g}_\alpha} \hat{u}_\alpha|^2 \leq C \hat{u}^\epsilon_\alpha \]

Now fixing \( \epsilon \) small such that \( \frac{2}{\epsilon} > \frac{n}{2} \) and applying the De Georgi-Nash-Moser iteration scheme gives for all \( p > 0 \) there exists \( C_p \) such that
\[ \sup_{B_0(R-1)} (\hat{u}_\alpha)^{1+\epsilon} \leq C_p\left(  \lVert ( \hat{u}_\alpha)^{1+\epsilon} \rVert_{L^p(B_0(R)} + \lVert  \hat{u}^\epsilon_\alpha \rVert_{L^{\frac{2}{\epsilon}}(B_0(R)} \right). \]
Setting \( p = \frac{2}{1+\epsilon} \) gives us
\[ \sup_{B_0(R-1)} (\hat{u}_\alpha)^{1+\epsilon} \leq C\lVert \hat{u}_\alpha \rVert_{L^2(B_0(R)}^\epsilon(1 + \lVert \hat{u}_\alpha \rVert_{L^2(B_0(R)}). \]
Because \( \hat{u}_\alpha(0) = 1 \), this implies \( \liminf_{\alpha \to \infty} \lVert \hat{u}_\alpha \rVert_{L^2(B_0(R))} > 0 \). However, because \( x_0 \) is the only concentration point we additionally have
\[ \lVert \hat{u}_\alpha \rVert_{L^2(B_0(R))} \leq C \left( \int_{B_0(R)} \hat{u}^{2^\#}_\alpha \right)^\frac{1}{2^\#} = C\left( \int_{B_{y_\alpha}(R\hat{\mu}_\alpha)} u_\alpha^{2^\#} \right)^\frac{1}{2^\#} \to 0.  \] and we have a contradiction.

Now let us consider the case (after potentially taking a subsequence) \( y_\alpha \to x_0 \). In order to apply the same logic as Step 1 in showing \( \hat{u}_\alpha \) converges to a limit function in \( C^{2m}(\mathbb{R}^n) \), we need to show \( (\Delta_{\hat{g}_\alpha} \hat{u}_\alpha)(0) \) is bounded from below. This corresponds to showing \( \hat{\mu}_\alpha^2 (\Delta_g u) (y_\alpha) \) is bounded from below. We let \( r_\alpha \) correspond to the function \( d(x_\alpha, \cdot) \). Because \( y_\alpha \to x_0 \), \( y_\alpha \) will be contained in an exponential chart around \( x_\alpha \) for \( \alpha \) sufficiently large. In particular, \( d_g(x_\alpha, x)^{\frac{n-2m}{2}} u_\alpha(x) \) is twice differentiable at \( y_\alpha \) for \( \alpha \) sufficiently large, so we have \( \Delta_g (r_\alpha^\frac{n-2m}{2} u_\alpha)(y_\alpha) \geq 0 \) and \( \nabla_g (r_\alpha^\frac{n-2m}{2} u_\alpha)(y_\alpha)=0 \). The latter equation implies \[ r_\alpha^\frac{n-2m}{2} \nabla_g u_\alpha(y_\alpha) = (- \nabla_g (r_\alpha^\frac{n-2m}{2}) u_\alpha)(y_\alpha). \] Then, using the well known formulas in an exponential chart \( |\nabla_g f(r)| = |f'(r)| \) and \( \Delta_g f(r) = \Delta_\xi f(r) - f'(r) \partial_r (\ln{\sqrt{g}}) \) for a radial function \( f \), we compute, with all expressions are evaluated at \( y_\alpha \),

\begin{align*} 0 \leq \Delta_g (r_\alpha^\frac{n-2m}{2} u_\alpha) =& \,  r_\alpha^\frac{n-2m}{2} \Delta_g u_\alpha + 2 \left(\frac{n-2m}{2}\right)^2 r_\alpha^\frac{n-2m-4}{2} u_\alpha\\
&- \left(\frac{n-2m}{2}\right)\left(\frac{3n-2m-4}{2}\right) r_\alpha^\frac{n-2m-4}{2}u_\alpha - \frac{n-2m}{2}r_\alpha^\frac{n-2m-2}{2}\partial_r (\ln{\sqrt{g}})u_\alpha \end{align*}
which implies
\[ \frac{n-2m}{2}r_\alpha^{-1}\partial_r (\ln{\sqrt{g}})u_\alpha + \left(\frac{n-2m}{2}\right)\left(\frac{n+2m-4}{2}\right) r_\alpha^{-2}u_\alpha \leq \Delta_g u_\alpha.  \]

The Cartan expansion of the metric shows \( \partial_r (\ln{\sqrt{g}}) \geq -Cr \) where \( C \) is based only on the curvature of the metric, therefore we have

\[ \frac{n-2m}{2}\left[ -C + \frac{n+2m-4}{2} r_\alpha^{-2} \right]u_\alpha \leq \Delta_g u_\alpha \]

The left hand side is bounded from below (positive for sufficiently large \( \alpha \) because \( y_\alpha \to x_0 \)), so our claim holds and we let \( \hat{u} : \mathbb{R}^n \to \mathbb{R} \) be the limit function of \( \hat{u}_\alpha \).

Let \( \delta > 0 \) be small and \( R \) be fixed. We recall notions from Step 1, using the same \( \tilde{u} \) and \( \tilde{u}_\alpha \),
\[ \int_{B_0(R)} \tilde{u}_\alpha^{2^\#} dv_{\tilde{g}_\alpha} = \int_{B_{x_\alpha}(R \mu_\alpha)}u^{2^\#}_\alpha dv_g \]
and
\[ \int_{\mathbb{R}^n} \tilde{u}^{2^\#} dx = 1. \]
We therefore have \[ \lim_{\alpha \to \infty} \int_{B_{x_\alpha}(R \mu_\alpha)}u^{2^\#}_\alpha dv_g = \int_{B_0(R)} \tilde{u}^{2^\#} dx =   1 - \epsilon_R \] where \( \epsilon_R \to 0 \) as \( R \to \infty \).

Combining this with the fact that \[ \lim_{\alpha \to \infty} \int_{B_{x_0}(\delta)} u_\alpha^{2^\#} dv_g = 1 \] immediately implies \[ \int_{B_{x_0}(\delta) \setminus B_{x_\alpha}(R \mu_\alpha)} u_\alpha^{2^\#} \leq \epsilon_R + o(1). \]
Because \( \hat{\mu}_\alpha \to 0 \) and \( y_\alpha \to x_0 \), for sufficiently large \( \alpha \) we obtain
\[ \int_{B_{y_\alpha} (\hat{\mu}_\alpha)  } u_\alpha^{2^\#} \leq \int_{B_{y_\alpha} (\hat{\mu}_\alpha) \cap B_{x_\alpha}(R \mu_\alpha)  } u_\alpha^{2^\#} + \epsilon_R + o(1). \]
As stated before \( \hat{\mu} = o(d_g(x_\alpha, y_\alpha)) \). Combining this fact with \( \mu_\alpha \leq \hat{\mu}_\alpha \) implies \( B_{y_\alpha} (\hat{\mu}_\alpha) \cap B_{x_\alpha}(R \mu_\alpha) = \emptyset \) for sufficiently large \( \alpha \). 
Therefore
\[ \int_{B_{y_\alpha} (\hat{\mu}_\alpha)  } u_\alpha^{2^\#} \leq \epsilon_R + o(1). \]
Taking the limit as \( \alpha \to \infty \) shows \[ \int_{B_0(1)} \hat{u}^{2^\#} dx \leq \epsilon_R. \]
However, the left hand side is a positive constant independent of \( R \) or \( \alpha \) and the right hand side converges to \( 0 \) as we send \( R \to \infty \) and we have a contradiction, therefore \( (\ref{step2}) \) holds.

As a consequence of Step 2, we prove for all \( 0 \leq k \leq m-1 \), \( v^{(k)}_\alpha \to 0 \) in \( C^0_{loc}(M \setminus{\{x_0\}}) \). By Step 2, \( u_\alpha \) is bounded in \( C^0_{loc}(M \setminus{\{x_0\}}) \) and \( v^{(m-1)}_\alpha \) satisfies \[ \Delta v^{(m-1)}_\alpha \leq \frac{1}{K} u^{2^\#-1}. \]
Integrating the PDE \( (\ref{originalpde}) \) implies \( \int \alpha^m u_\alpha \leq C \) by the divergence theorem, so for \( 0 \leq k \leq m-1 \), \[ \int_M v^{(k)}_\alpha = \int_M \alpha^{k} u_\alpha \to 0 \]
as \( \alpha \to \infty \).
Additonally, because \( \int_M u_\alpha \to 0 \) and \( u_\alpha \) is bounded in \( C^0_{loc}(M \setminus{\{x_0\}}) \), for any \( p>0 \) and for any compact subset \( E \) of \( M\setminus{\{x_0\}} \)
\[ \int_E u_\alpha^p \to 0 \]
as \( \alpha \to \infty \).
Therefore by the De Georgi-Nash-Moser iteration scheme  \( v^{(m-1)}_\alpha \to 0 \) in \( C^0_{loc}(M \setminus{\{x_0\}}) \). Now, assuming for some \( 1\leq k \leq m-1 \) \( v^{(k)}_\alpha \to 0 \) in \( C^0_{loc}(M \setminus{\{x_0\}}) \), then we consider
\[ \Delta_g v^{(k-1)}_\alpha \leq v^{(k)}_\alpha \]
and applying the De Georgi-Nash-Moser iteration scheme again shows \( v^{(k-1)}_\alpha \to 0 \) in \( C^0_{loc}(M \setminus{\{x_0\}}) \) and the claim follows by induction.

\textbf{Step 3: } \( L^2 \) concentration holds, i.e. for each \( \delta > 0 \) and for all \( 0 \leq i \leq m-1 \)

 \[\lim_{\alpha \to \infty} \frac{ \int_{M \setminus B_{x_0}(\delta)} |\Delta^\frac{i}{2}u_\alpha|^2 }{\int_M u_\alpha^2 } = 0\]

 Our first claim is that for all \( 0 \leq k \leq m-1 \), 
\begin{equation}\label{step3claim} \int_{M \setminus B_{x_0}( \delta)} (\Delta^\frac{k}{2} u_\alpha)^2 \leq C \alpha^k \int_{M \setminus B_{x_0}( 2^{-(2k+1)}\delta)} u^2 \end{equation}
Let \( \delta > 0 \) be arbitrary. Let us as before write \( v^{(m-1)}_\alpha = (\Delta + \alpha)^{m-1}u_\alpha \). By the last PDE in the system \( (\ref{systemwithouttilde}) \) we have \[ \Delta v_\alpha^{(m-1)} + \alpha v_\alpha^{(m-1)} - \frac{1}{\lambda_\alpha} u_\alpha^{2^{\#}-1} = 0.  \]
We write this as
\[ \sum_{i=1}^{m} c_{i,m} \alpha^{m-i} \Delta^i u_\alpha + \alpha^m u_\alpha  - \frac{1}{\lambda_\alpha} u_\alpha^{2^{\#}-1} = 0.\]
Because \( u_\alpha \to 0 \) in \( C^0_{loc}(M\setminus\{x_0\}) \), on any compact \( E \subset M\setminus\{x_0\} \) we have \( \alpha^m u_\alpha  - \frac{1}{\lambda_\alpha} u_\alpha^{2^{\#}-1} \geq 0 \) for suffiently large \( \alpha \). Additionally, by the simple identity \( (a+b)^k -b^k = a \sum_{i=0}^{k-1} b^i(a+b)^{k-1-i} \) we have
\[ \sum_{i=1}^{m} c_{i,m} \alpha^{m-i} \Delta^i u_\alpha  = \Delta\left(\sum_{i=0}^{m-1} \alpha^i v_\alpha^{(m-1-i)}\right). \]
Therefore, applying the  De Georgi-Nash-Moser scheme we obtain

\[ \sup_{M\setminus B_{x_0}( \delta)} \sum_{i=0}^{m-1} \alpha^i v_\alpha^{(m-1-i)} \leq C \int_{M \setminus B_{x_0}( \delta/2)} \sum_{i=0}^{m-1} \alpha^i v_\alpha^{(m-1-i)}\]

Let \( \eta \) be a smooth nonnegative function equal to \( 1 \) on \( M \setminus B_{\frac{\delta}{2}}(x_0)\) and equal to \( 0 \) on \( B_{\frac{\delta}{4}}(x_0)\). Then we continue the calculation, integrating by parts the expression \( \int_M \eta \sum_{i=0}^{m-1} \alpha^i v_\alpha^{(m-1-i)} \) to obtain

\begin{equation}\label{supbound}  \int_{M \setminus B_{x_0}( \delta/2)} \sum_{i=0}^{m-1} \alpha^i v_\alpha^{(m-1-i)} \leq \int_{M} \eta \sum_{i=0}^{m-1} \alpha^i v_\alpha^{(m-1-i)} \leq C\alpha^{m-1} \int_{M \setminus B_{x_0}( \delta/4)} u_\alpha.   \end{equation}

Let \( \eta \) now be a smooth nonnegative function equal to \( 1 \) on \( M \setminus B_{2\delta}(x_0)\) and equal to \( 0 \) on \( B_{\delta}(x_0)\). Now we prove \( (\ref{step3claim}) \) by induction. The base case \( k=0 \) holds by immediately. Now suppose for \( 1 \leq k \leq m-1 \) \( (\ref{step3claim}) \) is true for \( i < k \). Considering the \( m-1-k \)th term of \( \ref{supbound} \) individually and applying Cauchy-Schwarz gives us

\[ \alpha^{m-1-k} \int_M \eta v^{(k)}_\alpha u_\alpha \leq C \alpha^{m-1} \left(\int_{M \setminus B_{x_0}( \delta/4)} u_\alpha\right)^2 \leq C \alpha^{m-1} 
 \int_{M \setminus B_{x_0}( \delta/4)} u_\alpha^2   \]
 which we can write as
 \[  \int_M \eta u_\alpha  \sum_{l=0}^k \alpha^{k-l} c_{l,k} \Delta^l u_\alpha \leq C\alpha^{k} \int_{M \setminus B_{x_0}(\frac{\delta}{4})} u^2_\alpha .\]
 Applying Lemma A.6 to each term with \( r = \frac{\delta}{2} \) we obtain
 \begin{equation}\label{step3bound1} \int_M \eta \sum_{l=0}^k c_{l,k} \alpha^{k-l} (\Delta^\frac{l}{2} u_\alpha)^2 \leq C\alpha^{k} \int_{M \setminus B_{x_0}(\frac{\delta}{4})} u^2_\alpha + C\int_{M \setminus B_{x_0} (\delta/2)} \sum_{l=0}^k \alpha^{k-l} \sum_{p=0}^{l-1} (\Delta^\frac{p}{2} u_\alpha)^2. \end{equation}
 Given a fixed \( l \leq k \) for all \( 0 \leq p \leq l-1 \) we have by the induction hypothesis
 \begin{equation}\label{step3bound2} \int_{M \setminus B_{x_0} (\delta/2)} \alpha^{k-l} (\Delta^\frac{p}{2} u_\alpha)^2 \leq C\int_{M \setminus B_{x_0} (2^{-(2k)}\delta)} \alpha^{k-l} \alpha^{p} u_\alpha^2 \leq C \alpha^{k} \int_{M \setminus B_{x_0} (2^{-(2k)}\delta)} u_\alpha^2. \end{equation}
 We therefore obtain by \( (\ref{step3bound1}) \) and \( \ref{step3bound2}) \)
 \[ \int_{M \setminus B_{x_0}( 2\delta)} (\Delta^\frac{k}{2} u_\alpha)^2 \leq C \alpha^k \int_{M \setminus B_{x_0}( 2^{-(2k)}\delta)} u^2\]
 which is equivalent to \( (\ref{step3claim}) \). For all \( 1 \leq k \leq m-1 \) we obtain the bound
 \begin{equation}\label{step3boundreal} \int_{M \setminus B_{x_0}( \delta)} (\Delta^\frac{k}{2} u_\alpha)^2 \leq C \alpha^k \int_{M \setminus B_{x_0}( 2^{-(2m+1)}\delta)} u^2\end{equation}

Now multiplying the PDE \( (\ref{originalpde}) \) by \( \eta \) and using the fact that \( u_\alpha \to 0 \) in \( C^0_{\text{loc}}(M \setminus \{x_0\})\) we obtain

\[ \sum_{i=0}^m \int_M \eta c_{i, m}\alpha^{m-i} u_\alpha \Delta^i u_\alpha  = \int_M \eta u_\alpha^{2^{\#}} \leq \int_{M \setminus B_{x_0}( 2^{-(2m+1)}\delta)} u_\alpha^{2^\#} \leq  C\int_{M \setminus B_{x_0}( 2^{-(2m+1)}\delta)} u_\alpha^2 \]
We apply Lemma A.6 with \( r = \frac{\delta}{2} \) to each term and substitute \( i' = i-1 \) to obtain

\begin{align*} 
\sum_{i=0}^m \int_M \eta c_{i, m}\alpha^{m-i}  (\Delta^\frac{i}{2} u_\alpha)^2 & \leq  C \alpha^{m-1}\int_{M \setminus B_{x_0}(2^{-(2m+2)}\delta)} u^2_\alpha + C\sum_{i=1}^{m} \alpha^{m-i}\sum_{j=0}^{i-1} \int_{M\setminus B_{x_0(\frac{\delta}{2})}} (\Delta^j u_\alpha)^2 \\
&= C \alpha^{m-1}\int_{M \setminus B_{x_0}(2^{-(2m+2)}\delta)} u^2_\alpha +  C\sum_{i'=0}^{m-1} \alpha^{m-1-i'}\sum_{j=0}^{i'} \int_{M\setminus B_{x_0(\frac{\delta}{2})}} (\Delta^j u_\alpha)^2 \\
\end{align*}
Given some \( 0 \leq i' \leq m-1 \) and \( 0 \leq j \leq i' \) we apply \( (\ref{step3claim}) \) to obtain
\[ \alpha^{m-1-i'} \int_{M\setminus B_{x_0(\frac{\delta}{2})}} (\Delta^j u_\alpha)^2 \leq C\alpha^{m-1-i'+j}\int_{M \setminus B_{x_0}(2^{-(2m+2)}\delta)} u^2_\alpha \leq C \alpha^{m-1}\int_{M \setminus B_{x_0}(2^{-(2m+2)}\delta)} u^2_\alpha. \]
Hence we have
\begin{equation}\label{step3final} \sum_{i=0}^m \int_M \eta c_{i, m}\alpha^{m-i}  (\Delta^\frac{i}{2} u_\alpha)^2 \leq C \alpha^{m-1}\int_{M \setminus B_{x_0}(2^{-(2m+2)}\delta)} u^2_\alpha. \end{equation}
After dividing both sides by \( \alpha^{m-1} \) we have \begin{equation}\label{step3alpha} \alpha \int_{M \setminus B_{x_0}( 2\delta)} u^2_\alpha \leq C\int_{M \setminus B_{x_0}(2^{-(2m+2)}\delta)} u^2_\alpha \end{equation} and we have

\[  0 \leq \lim_{\alpha \to \infty} \frac{ \int_{M \setminus B_{x_0}(\delta)} u_\alpha^2 }{\int_M u_\alpha^2 } \leq \lim_{\alpha \to \infty} \frac{ \int_{M \setminus B_{x_0}(\delta)} u_\alpha^2 }{\int_{M \setminus B_{x_0}(2^{-(2m+3)}\delta)} u_\alpha^2}  \leq \lim_{\alpha \to \infty} \frac{C}{\alpha} = 0  \]
proving Step 3 for \( i=0 \). For \( 1 \leq i \leq m-1 \) we divide both sides of \((\ref{step3final})\) by \( \alpha^{m-i} \) and apply \( (\ref{step3alpha}) \) \( i-1 \) times to obtain
\[ \int_{M \setminus B_{x_0}(\delta)} (\Delta^\frac{i}{2}u_\alpha)^2 \leq C \alpha^{m-1-i}\int_{M \setminus B_{x_0}(2^{-(2m+3)}\delta)} u^2_\alpha \leq C \int_{M \setminus B_{x_0}(2^{-(i-1)(2m+3)}\delta)} u^2_\alpha  \] 
and we have
\[  0 \leq \lim_{\alpha \to \infty} \frac{ \int_{M \setminus B_{x_0}(\delta)} (\Delta^\frac{i}{2}u_\alpha)^2 }{\int_M u_\alpha^2 } \leq C\lim_{\alpha \to \infty} \frac{ \int_{M \setminus B_{x_0}(2^{-(i-1)(2m+3)}\delta)} u_\alpha^2 }{\int_M u_\alpha^2 }  =0 \]
and the proof of Step 3 is complete.

\subsection{Main Argument}

We first note an immediate consequence of Step 3 is for all \( 0 \leq k \leq m-1 \)
\begin{equation}\label{step3c} \lVert u_\alpha \rVert_{H_{k}^2(M)} \leq C \sum_{i=0}^{k} \int_{B_{x_0}(\delta)} (\Delta^{\frac{i}{2}} u_\alpha)^2 .\end{equation}

Now we let \( \frac{i_g}{4} > \delta > 0 \) and let \( \eta \) be a nonnegative smooth function defined on \( \mathbb{R}^n \) such that \( \eta = 1\) on \( B_0(2\delta) \) and equal to \( 0 \) on \( \mathbb{R}^n \setminus B_0(3\delta) \). We define \( \eta_\alpha \) on \( M \) by \( \eta_\alpha(x) = \eta(\text{exp}_{x_\alpha}^{-1}(x)) \). We then define \( U_\alpha \) by \( U_\alpha = \eta_\alpha u_\alpha \). We have uniform bounds for all derivatives of \( \eta_\alpha \).

Because \( U_\alpha \) is only nonzero inside a geodesic chart, we consider the Euclidean metric \( \xi \) on \( B_{x_\alpha}(3\delta) \) defined through the pullback of \( exp^{-1}_{x_\alpha} \). We write \( dx = dv_\xi \) and the Euclidean Sobolev inequality gives

\begin{equation}\label{euclid} \left(\int_M U_\alpha^{2^\#} dx\right)^\frac{2}{2^\#} \leq K \int_M(\Delta_\xi^\frac{m}{2} U_\alpha)^2 dx \end{equation}

Euclidean integration by parts shows
\[ \int_M (\Delta_\xi^\frac{m}{2} U_\alpha)^2 dx = \int_M |\nabla^m_\xi U_\alpha|^2 dx .\] 
Then Lemma A.9 gives
\[  \int_M |\nabla^m_\xi U_\alpha|^2 dx \leq \int_M (1 + C r_\alpha^2)\left(|\nabla^m_g U_\alpha|^2 + C \sum_{i=0}^{m-1}|\nabla^i_g U_\alpha|^2 \right) dv_g \]
where \( r_\alpha = d_g(x_\alpha, \cdot) \).
Then by Lemma A.3 \begin{equation}\label{eucaboveterms} \int_M (\Delta_\xi^\frac{m}{2} U_\alpha)^2 dx \leq \int_M (\Delta_g^\frac{m}{2} U_\alpha)^2 dv_g + C \int_M r_\alpha^2 |\nabla_g^m U_\alpha|^2 dv_g + C \sum_{i=0}^{m-1}  \int_M |\nabla_g^i U_\alpha|^2 dv_g \end{equation}
We consider the third term in \( (\ref{eucaboveterms}) \). Applying norm equivalence for \( H_m^2(M) \) (see Robert\cite{robert}), Lemma A.5 and \( (\ref{step3c}) \) results in
\begin{equation}\label{remainderU} \sum_{i=0}^{m-1}  \int_M  |\nabla_g^i U_\alpha|^2 dv \leq C\sum_{i=0}^{m-1}\int_M  |\Delta_g^\frac{i}{2} U_\alpha|^2 dv  \leq C\sum_{i=0}^{m-1}\int_M  |\Delta_g^\frac{i}{2} u_\alpha|^2 dv \leq  C \sum_{i=0}^{m-1} \int_{B_{x_0}(\delta)} (\Delta^{\frac{i}{2}} u_\alpha)^2  \end{equation}
where Lemma A.5 is applied to each individual term of the second sum.
For the first term of \( (\ref{eucaboveterms}) \), we use Lemma A.5, A.6 and \( (\ref{step3c}) \) to obtain 
\begin{align*} \int_M (\Delta_g^\frac{m}{2} U_\alpha)^2 & \leq \int_M\eta_\alpha^2 (\Delta_g^\frac{m}{2} u_\alpha)^2 + C\sum_{i=0}^{m-1} \int_{B_{x_0}(\delta)} (\Delta^{\frac{i}{2}} u_\alpha)^2 \\  
& \leq \int_M\eta_\alpha^2 u_\alpha \Delta^m u_\alpha + C\sum_{i=0}^{m-1} \int_{B_{x_0}(\delta)} (\Delta^{\frac{i}{2}} u_\alpha)^2 \end{align*}
We then substitute \( \Delta^m u_\alpha = u^{2^\#-1}_\alpha - \sum_{i=0}^{m-1} c_{i,m} \alpha^{m-i} \Delta^i u_\alpha\), apply Lemma A.6 to each term, and apply \( (\ref{step3c}) \) to obtain
\begin{align} \int_M (\Delta_g^\frac{m}{2} U_\alpha)^2 & \leq \lambda_\alpha \int_M\eta_\alpha^2 u^{2^\#}_\alpha - \left( \sum_{i=0}^{m-1} c_{i,m}\alpha^{m-i}\int_M\eta_\alpha^2 u_\alpha \Delta^i u_\alpha \right) + C\sum_{i=0}^{m-1} \int_{B_{x_0}(\delta)} (\Delta^{\frac{i}{2}} u_\alpha)^2 \label{boundbegin} \\
& \leq \frac{1}{K} \int_M\eta_\alpha^2 u^{2^\#}_\alpha - \left( \sum_{i=0}^{m-1} c_{i,m}\alpha^{m-i}\int_M\eta_\alpha^2  (\Delta^\frac{i}{2} u_\alpha)^2 \right) \notag \\
& \quad\quad\quad\quad\quad\quad\quad+ C\sum_{i=1}^{m-1} \alpha^{m-i}\sum_{j=0}^{i-1} \int_M (\Delta^\frac{j}{2} u)^2 + C\sum_{i=0}^{m-1} \int_{B_{x_0}(\delta)} (\Delta^{\frac{i}{2}} u_\alpha)^2 \\
& \leq \frac{1}{K} \int_M\eta_\alpha^2 u^{2^\#}_\alpha - \left( \sum_{i=0}^{m-1} c_{i,m}\alpha^{m-i}\int_M\eta_\alpha^2  (\Delta^\frac{i}{2} u_\alpha)^2 \right) \notag \\
& \quad\quad\quad\quad\quad\quad\quad+ C\sum_{i=1}^{m-1} \alpha^{m-i}\sum_{j=0}^{i-1} \int_{B_{x_0}(\delta)} (\Delta^\frac{j}{2} u)^2 + C\sum_{i=0}^{m-1} \int_{B_{x_0}(\delta)} (\Delta^{\frac{i}{2}} u_\alpha)^2.
\end{align}

Considering the remainder terms, substituting \( i' = i-1 \) we have
\begin{align} \sum_{i=1}^{m-1} \alpha^{m-i}\sum_{j=0}^{i-1} \int_{B_{x_0}(\delta)} (\Delta^\frac{j}{2} u_\alpha)^2 + \sum_{i=0}^{m-1} \int_{B_{x_0}(\delta)} (\Delta^{\frac{i}{2}} u_\alpha)^2 &= \sum_{i'=0}^{m-2} \alpha^{m-1-i'}\sum_{j=0}^{i'} \int_{B_{x_0}(\delta)} (\Delta^\frac{j}{2} u_\alpha)^2 + \sum_{i=0}^{m-1} \int_{B_{x_0}(\delta)} (\Delta^{\frac{i}{2}} u_\alpha)^2 \\
&= \sum_{i'=0}^{m-1} \alpha^{m-1-i'}\sum_{j=0}^{i'} \int_{B_{x_0}(\delta)} (\Delta^\frac{j}{2} u_\alpha)^2 
\end{align}
Then reindexing with \( (\ref{5.r}) \) gives 
\begin{align} \sum_{i'=0}^{m-1} \alpha^{m-1-i'}\sum_{j=0}^{i'} \int_{B_{x_0}(\delta)} (\Delta^\frac{j}{2} u_\alpha)^2 &= \sum_{i=0}^{m-1} \sum_{j=0}^{m-1-i} \int_{B_{x_0}(\delta)} \alpha^{m-1-i-j} (\Delta^\frac{j}{2} u_\alpha)^2\\
&= \sum_{j=0}^{m-1} \sum_{i=0}^{m-1-j} \int_{B_{x_0}(\delta)} \alpha^{m-1-j-i} (\Delta^\frac{j}{2} u_\alpha)^2 \\
&=  \sum_{j=0}^{m-1}  \int_{B_{x_0}(\delta)} (\Delta^j u_\alpha)^2 \sum_{i=0}^{m-1-j}  \alpha^{i} \label{boundend} .
\end{align}
Therefore we obtain, using the fact that \( B_{x_\alpha}(2\delta) \supset B_{x_0}(\delta) \) for large \( \alpha \),
\begin{equation}\label{poly} \int_M (\Delta_g^\frac{m}{2} U_\alpha)^2 \leq \frac{1}{K} \int_M u^{2^\#}_\alpha -   \sum_{i=0}^{m-1} \left(c_{i,m} \alpha^{m-i} - \sum_{j=0}^{m-1-i} C\alpha^j \right)\int_{B_{x_0}(\delta)} (\Delta^{\frac{i}{2}} u)^2  .\end{equation}

We now estimate the second term of \( (\ref{eucaboveterms}) \). Applying Lemmas A.3, A.5, A.6, substituting the PDE \( (\ref{originalpde}) \), and applying \( (\ref{remainderU}) \) gives us

\begin{align} \int_M r_\alpha^2 |\nabla_g^m U_\alpha|^2 dv & \leq \int_M r_\alpha^2 (\Delta^{\frac{m}{2}}U_\alpha)^2 + C\lVert U_\alpha \rVert_{H_{m-1}^2} \notag \\
& \leq \int_M \eta_\alpha^2 r_\alpha^2 (\Delta^{\frac{m}{2}}u_\alpha)^2 + C\lVert U_\alpha \rVert_{H_{m-1}^2} \notag \\
& \leq \int_M \eta_\alpha^2 r_\alpha^2 u_\alpha \Delta^m u_\alpha + C \lVert U_\alpha \rVert_{H_{m-1}^2} \notag \\
& \leq \frac{1}{K} \int_M \eta_\alpha^2 r_\alpha^2 u^{2^{\#}}_\alpha dv_g - \left( \sum_{i=0}^{m-1} c_{i,m}\alpha^{m-i}\int_M\eta_\alpha^2 r_\alpha^2 u_\alpha \Delta^i u_\alpha \right) + C\sum_{i=0}^{m-1} \int_{B_{x_0}(\delta)} (\Delta^{\frac{i}{2}} u_\alpha)^2 \notag  \\
&\leq \frac{1}{K} \int_M \eta_\alpha^2 r_\alpha^2 u^{2^{\#}}_\alpha dv_g + C \sum_{j=0}^{m-1}  \int_{B_{x_0}(\delta)} (\Delta^j u_\alpha)^2 \sum_{i=0}^{m-1-j}  \alpha^{i} \label{secondtermbound1}
\end{align}
where to obtain \( (\ref{secondtermbound1}) \) we have applied the computations from \( (\ref{boundbegin})-(\ref{boundend}) \).

By Step 2, we have 
\[ r_\alpha^2 u^{2^{\#}}_\alpha = u_\alpha (r_\alpha u^{\frac{n}{n-2m}}) (r_\alpha u_\alpha^{\frac{2}{n-2m}}) u^{\frac{2m-2}{n-2m}} \\
 \leq C u_\alpha (r_\alpha u^{\frac{n}{n-2m}}_\alpha) u^{\frac{2m-2}{n-2m}}_\alpha \]
 and therefore \[ \eta_\alpha^2 r_\alpha^2 u^{2^{\#}}_\alpha \leq C \eta_\alpha u_\alpha (r_\alpha \eta_\alpha u^{\frac{n}{n-2m}}_\alpha) u^{\frac{2m-2}{n-2m}}_\alpha. \]
 Because \( \frac{n - 2(m-1)}{2n} + \frac{1}{2} + \frac{2m-2}{2n} = 1 \), we apply Hölder's inequality to the right hand side to obtain
\[ \int_M \eta_\alpha^2 r_\alpha^2 u^{2^{\#}}_\alpha \leq C \left(\int_M (\eta_\alpha u_\alpha)^{\frac{2n}{n - 2(m-1)}}dv_g\right)^{\frac{n-2(m-1)}{2n}}\left( \int_M \eta_\alpha^2 r_\alpha^2 u^{2^{\#}}_\alpha \right)^\frac{1}{2}\left(\int_M u_\alpha^{2^{\#}}\right)^{\frac{2m-2}{2n}}. \]
Because \( \int_M u_\alpha^{2^{\#}} = 1 \), this gives us
\begin{equation}\label{boundingetar} \int_M \eta_\alpha^2 r_\alpha^2 u^{2^{\#}}_\alpha \leq C \left(\int_M (\eta_\alpha u_\alpha)^{\frac{2n}{n - 2(m-1)}}\right)^{\frac{n-2(m-1)}{n}}.  \end{equation}
We then apply the Sobolev embedding theorem and \( (\ref{remainderU}) \) to achieve
\begin{equation}\label{boundingetar2} \left(\int_M (\eta_\alpha u_\alpha)^{\frac{2n}{n - 2(m-1)}}\right)^{\frac{n-2(m-1)}{n}} \leq C \lVert u_\alpha \rVert_{H_{m-1}^2} \leq C \sum_{i=0}^{m-1} \int_{B_{x_0}(\delta)} (\Delta^{\frac{i}{2}} u_\alpha)^2 \end{equation}
Therefore, by putting together \( (\ref{secondtermbound1}),(\ref{boundingetar}), \) and \( (\ref{boundingetar2}) \) we obtain \begin{equation}\label{eucaboveradius} \int_M r_\alpha^2 |\nabla_g^m U|^2 dv \leq C \sum_{i=0}^{m-1}  \int_{B_{x_0}(\delta)} \sum_{j=0}^{m-1-i}  \alpha^{i} (\Delta^\frac{i}{2} u_\alpha)^2 \end{equation}
Hence, by combining \( (\ref{eucaboveterms}), (\ref{remainderU}), (\ref{poly}),(\ref{eucaboveradius}) \) we obtain
\begin{equation}\label{euclidabove} K \int_M (\Delta_\xi^\frac{m}{2} U_\alpha)^2 dx \leq  \int_M u^{2^\#}_\alpha -  \sum_{i=0}^{m-1} \left(K c_{i,m}\alpha^{m-i} - \sum_{j=0}^{m-i-1} C\alpha^j \right)\int_{B_{x_0}(\delta)} (\Delta^{\frac{i}{2}} u_\alpha)^2 .  \end{equation}

Now we consider the left hand side of \( (\ref{euclid}) \). By the Cartan expansion of the metric and the definition of \( U_\alpha \) we have
\[
\int_M U_\alpha^{2^\#} dx  \geq \int_{B_{x_0}(\delta)} u_\alpha^{2^\#} dv_g -C\int_M r_\alpha^2 U_\alpha^{2^\#} dv_g . \\
\]
Since \( 0 \leq \eta_\alpha \leq 1 \) we must have \( \eta^{2^\#}_\alpha \leq \eta^2_\alpha \) and so by \( (\ref{boundingetar}) \) and \( (\ref{boundingetar2}) \) we have
\[ \int_M \eta_\alpha^{2^\#} r_\alpha^2 u^{2^{\#}}_\alpha \leq \int_M \eta_\alpha^2 r_\alpha^2 u^{2^{\#}}_\alpha  
 \leq C \sum_{i=0}^{m-1} \int_{B_{x_0}(\delta)} (\Delta^{\frac{i}{2}} u_\alpha)^2. \]
 Therefore we obtain
 \begin{equation}\label{firstsecondterm} \int_M U_\alpha^{2^\#} dx  \geq \int_{B_{x_0}(\delta)} u_\alpha^{2^\#} dv_g - C  \sum_{i=0}^{m-1} \int_{B_{x_0}(\delta)} (\Delta^{\frac{i}{2}} u_\alpha)^2\end{equation}
The first term on the right side of \( (\ref{firstsecondterm}) \) to \( 1 \) from below and the second term converges to \( 0 \), therefore raising both sides to \( \frac{2}{2^\#} < 1 \) we have for sufficiently large \( \alpha \)
\begin{equation}\label{previousinequality} \left(\int_M U_\alpha^{2^\#} dx\right)^\frac{2}{2^\#} \geq \int_{B_{x_0}(\delta)} u^{2^\#}_\alpha - C  \sum_{i=0}^{m-1} \int_{B_{x_0}(\delta)} (\Delta^{\frac{i}{2}} u_\alpha)^2  \end{equation}

Independently we have by Step 3 \[ \int_{M \setminus B_{x_0}(\delta)} u^{2^\#}_\alpha \leq \left(\sup_{M \setminus B_{x_0}(\delta)} u_\alpha\right)^{2^\#-2} \int_{M \setminus B_{x_0}(\delta)} u^2_\alpha \leq C \int_{B_{x_0}(\delta)} u_\alpha^2  \]
allowing us to strengthen \( (\ref{previousinequality}) \) to 
\begin{equation}\label{euclidbelow} \left(\int_M U_\alpha^{2^\#} dx\right)^\frac{2}{2^\#} \geq \int_M u^{2^\#}_\alpha - C  \sum_{i=0}^{m-1} \int_{B_{x_0}(\delta)} (\Delta^{\frac{i}{2}} u_\alpha)^2 \end{equation}

Therefore combining \( (\ref{euclid})\), \( (\ref{euclidabove})\), and \( (\ref{euclidbelow})\) we obtain
\[  \sum_{i=0}^{m-1} \left(K c_{i,m}\alpha^{m-i} - \sum_{j=0}^{m-i-1} C\alpha^j \right)\int_{B_{x_0}(\delta)} (\Delta^{\frac{i}{2}} u_\alpha)^2  \leq 0 \]
and we obtain a contradiction when \( \alpha \) is sufficiently large. \qed

\appendix
\section{Technical Lemmas}

In this section \( (M, g) \) will represent some smooth complete \( n \)-dimensional Riemannian manifold without boundary with bounded curvature.

In the following, given two tensors \( A, B \) we adopt the notation \( A \star B \) to denote a linear combination of contractions of \( A \otimes B \), possibly after raising and lowering indices using the metric and including the trivial linear combination \( 0 \cdot A \otimes B \). We write \( A \star_k B \) in the case each element of the linear combination is covariant of degree \( k \). 
Given a permutation \( \sigma \) of \( \{1, \dots, k \} \) and a covariant k-tensor \( A \) we define \( \sigma \cdot A \) by \( (\sigma \cdot A)_{i_1 \dots i_k} = A_{i_{\sigma(1)} \dots i_{\sigma(k)}} \). We use the notation \( A\bigstar_k B\) to denote a sum of the form
\[ \sum_{i=0}^m \sigma_i \cdot (A \star_k B) \]
for some \( \{ \sigma_i: 1 \leq i \leq m\} \) permutations on \( \{1, \dots, k \} \). We have the relation \( \nabla(A \bigstar_k B) = \nabla A \bigstar_{k+1} B + A \bigstar_{k+1} \nabla B \).

All integrals will be assumed to be with respect to the Riemannian volume element \( dv_g \) unless explicitly stated otherwise. All Laplacians will be assumed to be with respect to the metric unless explicitly stated otherwise. We let \( R \) represent the Riemann curvature tensor.

\begin{lem}
    Let \( k \geq 0 \) be an integer and \( u \in H_k^2(M) \). Let \( \sigma \) be a permutation of \( \{1, \dots, k\} \).  
    Then 
    \[ \nabla^k u - \sigma \cdot \nabla^k u = \sum_{0 \leq l \leq k-3}  \nabla^l R  \bigstar_k   \nabla^{k-2-l} u . \]
    
\end{lem}

\begin{proof}
Because the statement is immediate for \( k = 0,1,2 \), we operate under the assumption \( k \geq 3 \). We first consider what happens when \( \sigma \) is a transposition of two consecutive elements \( \cycle{j,j+1} \) where \( j \leq k-1 \). In this case, we have
\[ \nabla_{i_1} \dots \nabla_{i_{j-1}}(\nabla_{i_{j}} \nabla_{i_{j+1}}-\nabla_{i_{j+1}} \nabla_{i_j}) \nabla_{i_{j+2}} \dots \nabla_{i_k}u = \nabla_{i_1} \dots \nabla_{i_{j-1}}(\sum_{l=j+2}^{k} R_{i_j i_{j+1} i_l}^\alpha \nabla_{i_{j+2}} \dots \nabla_{i_{l-1}} \nabla_\alpha \nabla_{i_{l +1}} \dots \nabla_{i_k} u) \]
This shows in this case by a Leibniz rule applied to covariant differentiation (see Gavrilov\cite{gav}) that
\[ \nabla^k u - \sigma \cdot \nabla^k u = \sum_{0 \leq l \leq k-3}  \nabla^l R  \bigstar_k   \nabla^{k-2-l} u. \] Now for general \( \sigma \), we write \( \sigma = \tau_q \dots \tau_1\) where \( \tau_i \) is a transposition of consecutive elements. Let us also write \( \sigma_p = \tau_p \dots \tau_1 \) for \( 1 \leq p \leq q \). Then we write \( \nabla^k u - \sigma \cdot \nabla^k u \) as a telescoping sum
\[ \nabla^k u - \sigma_1 \nabla^k u + \sigma_1 \nabla^k u - \sigma_2 \nabla^k u \dots - \sigma_{q-1} \nabla^k u + \sigma_{q-1}\cdot \nabla^k u - \sigma \cdot \nabla^k u  \]
and therefore by applying the case of a transposition to each difference we obtain
\[ \nabla^k u - \sigma \cdot \nabla^k u = \sum_{0 \leq l \leq k-3}  \nabla^l R  \bigstar_k   \nabla^{k-2-l} u . \]

\end{proof}

While Lemmas A.2-A.9 are stated for functions in \( u \in H_k^2(M) \), by density it will suffice to prove them for \( u \in C_c^\infty(M) \), therefore all functions from now on will be assumed smooth.

\begin{lem}
Let \( k \geq 0 \) be an integer. Let 
\[ (k_1, k_2) \in \{ (k, k), (k-1, k+1), (k-1, k)  \}. \]
For each \( i \in \{1, 2 \} \), let \( j_i \leq \frac{k_i}{2} \), let \( T_{j_i} \) be the operator taking a \( k_i \) degree tensor and contracting on the last \( j_i \) pairs of indices, i.e. \( T_{j_i}(\nabla^{k_i} u) = \nabla^{k_i - 2j_i}\Delta^{j_i} u \), and let \( \sigma_i, \sigma_i' \) be permutations on \( \{1, \dots k_i\} \). Let \( S \) be an arbitrary compactly supported tensor of degree \( q \) such that \( q \geq 2|(j_2 - \frac{k_2-k_1}{2})-j_1| \). For each \( k_2 - 2j_2 \) degree covariant tensor \( A \), let \( A \smwhitestar S \) denote some fixed contraction (possibly after raising and lowering indices using the metric) of \( A \otimes S \) of degree \( k_1 - 2 j_1 \) where all contractions either occur within \( S \) or take one index from each of \( A \) and \( S \) i.e. no contractions occur within \( A \). Then there exists some \( C \) based on \( n, k, \max |S|, \max |\nabla S|,\max |\nabla^2 S| \) and bounds for \( R \) and finitely many of its derivatives such that for all \( u \in H_{k_2}^2(M) \),
    \begin{equation}\label{1.2.2} \left| \int_M \langle T_{j_1}(\sigma_1 \cdot \nabla^{k_1} u) , T_{j_2}(\sigma_2 \cdot \nabla^{k_2} u) \smwhitestar S \rangle - \int_M \langle T_{j_1}(\sigma_1' \cdot \nabla^{k_1} u) , T_{j_2}(\sigma_2' \cdot \nabla^{k_2} u) \smwhitestar S \rangle \right| \leq  C \sum_{i=0}^{k-1} \int_{supp(S)} |\nabla^i u|^2  \end{equation}
\end{lem}

\begin{proof} 
We note \( \smwhitestar \) immediately grants the existence of some operator \(  \smwhitestar' \) such that for each \( k_1-2j_1 \) degree covariant tensor \( B \), \(  B \smwhitestar' S \) is some contraction of \( B \otimes S \) of degree \( k_2 - 2j_2 \) (where all contractions are within \( S \) or take one index from \( B \) and one index from \( S \)) such that for all \( k_2-2j_2\) degree covariant tensors \( A \), \[ \langle B, A \smwhitestar S \rangle = \langle B \smwhitestar' S, A \rangle. \]
We prove the statement in the case \( \sigma_1'=\sigma_2' = \text{Id} \), the full statement then follows by a simple application of the triangle inequality.

We have
 \[  \int_M \langle T_{j_1}(\sigma_1 \cdot \nabla^{k_1} u) , T_{j_2}(\sigma_2 \cdot \nabla^{k_2} u) \smwhitestar S \rangle - \int_M \langle T_{j_1}(\nabla^{k_1} u) , T_{j_2}(  \nabla^{k_2} u) \smwhitestar S \rangle  \] \[ = \int_M \langle T_{j_1}(\sigma_1 \cdot \nabla^{k_1} u) , T_{j_2}(\sigma_2 \cdot \nabla^{k_2} u) \smwhitestar S \rangle - \int_M \langle T_{j_1}(\sigma_1 \cdot \nabla^{k_1} u) , T_{j_2}( \nabla^{k_2} u) \smwhitestar S \rangle  \] \[ + \int_M \langle T_{j_1}(\sigma_1 \cdot \nabla^{k_1} u) , T_{j_2}( \nabla^{k_2} u) \smwhitestar S \rangle - \int_M \langle T_{j_1}( \nabla^{k_1} u) , T_{j_2}( \nabla^{k_2} u) \smwhitestar S \rangle.   \]
 Then considering the first difference we have by applying Lemma A.1 and extending our \( \bigstar \) notation such that instances of \( A \bigstar_k^{(i)} B \) represent fixed choices for \( A \bigstar_k B \),
 \begin{align*}
 &\int_M \langle T_{j_1}(\sigma_1 \cdot \nabla^{k_1} u) , T_{j_2}(\sigma_2 \cdot \nabla^{k_2} u) \smwhitestar S \rangle - \int_M \langle T_{j_1}(\sigma_1 \cdot \nabla^{k_1} u) , T_{j_2}( \nabla^{k_2} u) \smwhitestar S \rangle \\
 = &\int_M \langle T_{j_1}(\sigma_1 \cdot \nabla^{k_1} u), T_{j_2} (\sigma_2 \cdot \nabla^{k_2} u - \nabla^{k_2} u) \smwhitestar S  \rangle \\
 = &\int_M \langle T_{j_1}(\sigma_1 \cdot \nabla^{k_1} u), T_{j_2} \left( \sum_{0 \leq l \leq k_2-3}  \nabla^l R  \bigstar_{k_2}^{(1)}   \nabla^{k_2-2-l} u\right) \smwhitestar S  \rangle \\
     = &\int_M \langle T_{j_1}(\sigma_1 \cdot \nabla^{k_1} u), T_{j_2} \left( \sum_{0 \leq l \leq k_2-3}  \nabla^l R  \bigstar_{k_2}^{(1)}   \nabla^{k_2-2-l} u\right) \smwhitestar S  \rangle 
     - \langle T_{j_1} \nabla^{k_1} u, T_{j_2} \left( \sum_{0 \leq l \leq k_2-3}  \nabla^l R  \bigstar_{k_2}^{(1)}   \nabla^{k_2-2-l} u\right) \smwhitestar S \rangle \\
    & \quad\quad\quad\quad\quad\quad\quad\quad\quad\quad\quad+ \int_M \langle T_{j_1} \nabla^{k_1} u, T_{j_2} \left( \sum_{0 \leq l \leq k_2-3}  \nabla^l R  \bigstar_{k_2}^{(1)}   \nabla^{k_2-2-l} u\right)\smwhitestar S \rangle \\
     = &\int_M \langle T_{j_1} \left( \sum_{0 \leq l \leq k_1-3}  \nabla^l R  \bigstar_{k_1}^{(2)}   \nabla^{k_1-2-l} u\right) ,T_{j_2} \left( \sum_{0 \leq l \leq k_2-3}  \nabla^l R  \bigstar_{k_2}^{(1)}   \nabla^{k_2-2-l} u\right) \smwhitestar S \rangle 
     \\
     &\quad\quad\quad\quad\quad\quad\quad\quad\quad\quad\quad+ \int_M \langle \nabla^{k_1-2j_1}  \Delta^{j_1} u,T_{j_2} \left( \sum_{0 \leq l \leq k_2-3}  \nabla^l R  \bigstar_{k_2}^{(1)}   \nabla^{k_2-2-l} u\right) \smwhitestar S \rangle. \\
 \end{align*}
We clearly have \[ \left|\int_M \langle T_{j_1} \left( \sum_{0 \leq l \leq k_1-3}  \nabla^l R  \bigstar_{k_1}^{(2)}   \nabla^{k_1-2-l} u\right) ,T_{j_2} \left( \sum_{0 \leq l \leq k_2-3}  \nabla^l R  \bigstar_{k_2}^{(1)}   \nabla^{k_2-2-l} u\right) \smwhitestar S \rangle\right| \leq C \sum_{i=0}^{k-1} \int_{supp(S)} |\nabla^i u|^2.  \] 
If \( (k_1, k_2) \in \{ (k-1, k+1), (k-1, k) \} \), then we also clearly have 
\[ \left| \int_M \langle \nabla^{k_1-2j_1}  \Delta^{j_1} u,T_{j_2} \left( \sum_{0 \leq l \leq k_2-3}  \nabla^l R  \bigstar_{k_2}^{(1)}   \nabla^{k_2-2-l} u\right) \smwhitestar S \rangle \right| \leq C \sum_{i=0}^{k-1} \int_{supp(S)} |\nabla^i u|^2dv_g. \]
In the case \( (k_1, k_2) = (k,k) \), we need to integrate by parts. If \( 2j_1 = k \) then after integrating by parts and applying the Leibniz rule, we obtain
\begin{align*} 
& \left| \int_M \langle \nabla^{k_1-2j_1}  \Delta^{j_1} u,T_{j_2} \left( \sum_{0 \leq l \leq k_2-3}  \nabla^l R  \bigstar_{k_2}^{(1)}   \nabla^{k_2-2-l} u\right) \smwhitestar S \rangle \right| \\
= &\left| \int_M \langle  \Delta^{k} u,\left( \sum_{0 \leq l \leq k-3}  \nabla^l R  \bigstar_{k-2j_2}   \nabla^{k-2-l} u\right) \smwhitestar S \rangle \right| \\
\leq & \left| \int_M \langle  \nabla \Delta^{k-1} u ,\left( \sum_{0 \leq l \leq k-2}  \nabla^l R  \bigstar_{k-2j_2+1}   \nabla^{k-1-l} u \right) \bigstar_{1} S \rangle \right|   \\
& \quad\quad\quad+  \left| \int_M \langle \nabla \Delta^{k-1} u ,\left( \sum_{0 \leq l \leq k-3}  \nabla^l R\bigstar_{k-2j_2}   \nabla^{k-2-l} u \right) \bigstar_1 \nabla S \rangle  \right| \\
\leq & \, C \sum_{i=0}^{k-1} \int_{supp(S)} |\nabla^i u|^2. \end{align*}
If \( 2j_1 < k \) then a similar computation gives
\begin{align*} &\left| \int_M \langle \nabla^{k-2j_1}  \Delta^{j_1} u,T_{j_2} \left( \sum_{0 \leq l \leq k-3}  \nabla^l R  \bigstar_{k}^{(1)}   \nabla^{k-2-l} u\right) \smwhitestar S \rangle \right| \\
= &\left| \int_M \langle \nabla^{k-2j_1 - 1}  \Delta^{j_1} u, div \left( \left(  \sum_{0 \leq l \leq k-3}  \nabla^l R  \bigstar_{k-2j_2}   \nabla^{k-2-l} u\right) \smwhitestar S \right) \rangle \right|\\
 \leq & \, C \sum_{i=0}^{k-1} \int_{supp(S)} |\nabla^i u|^2 \end{align*}

We also have, applying Lemma A.1,
\begin{align*}
   & \left| \int_M \langle T_{j_1}(\sigma_1 \cdot \nabla^{k} u) , T_{j_2}( \nabla^{k} u) \smwhitestar S \rangle - \int_M \langle T_{j_1}( \nabla^{k} u) , T_{j_2}( \nabla^{k} u) \smwhitestar S \rangle \right| \\ 
   = \, \, & \left| \int_M \langle T_{j_1} \left( \sum_{0 \leq l \leq k-3}  \nabla^l R  \bigstar_k   \nabla^{k-2-l} u\right), T_{j_2}( \nabla^{k} u) \smwhitestar S\rangle \right| \\
   = \, \, & \left| \int_M \langle T_{j_1} \left( \sum_{0 \leq l \leq k-3}  \nabla^l R  \bigstar_k   \nabla^{k-2-l} u\right) \smwhitestar' S, \nabla^{k-2j_2}  \Delta^{j_2} u  \rangle \right| \\
   \leq \, \, &C \sum_{i=0}^{k-1} \int_{supp(S)} |\nabla^i u|^2 
\end{align*}
after integrating by parts similarly to above. Therefore we have obtained \( (\ref{1.2.2}) \)

\end{proof}

\begin{lem} Let \( k \geq 0 \) be an integer and \( \eta \) be an arbitrary compactly supported smooth function.  Then there exists \( C \), based on \( n, k\), and bounds for \( R \) and finitely many of its derivatives, bounds for \( \eta \) and finitely many of its derivatives such that for all \( u \in H_k^2(M) \), 
\[ \left| \int_M \eta | \nabla^k u |^2 - \int_M \eta (\Delta^{\frac{k}{2}} u)^2 \right| \leq C  \sum_{i=0}^{k-1} \int_{supp(\eta)} |\nabla^i u|^2   \]
.
\end{lem}

\begin{proof}
Throughout the proof we use the notation \( A \equiv B \) to mean \[|A - B| \leq C  \sum_{i=0}^{k-1} \int_{supp(\eta)} |\nabla^i u|^2. \] This clearly satisfies the assumptions of an equivalence relation.
We therefore perform the proof of Lemma A.3 by finding expressions \( A_1 \dots A_p \) such that 
\[ \int_M \eta | \nabla^k u |^2 \equiv A_1 \dots \equiv A_p \equiv \int_M \eta (\Delta^{\frac{k}{2}} u)^2. \]

We prove by induction. The statement is immediate for \( k = 0 \) and \( k = 1 \). Now let \( k \geq 2 \) and assume the statement holds true for \( k-2 \). We apply Lemma A.2 to obtain
 \begin{equation}\label{3.1} \int_M \eta | \nabla^k u |^2 = \int_M \eta \nabla^{i_1} \dots \nabla^{i_k} u \nabla_{i_1} \dots \nabla_{i_k} u \equiv \int_M \eta \nabla^{i_k} \nabla^{i_1} \dots \nabla^{i_{k-1}} u \nabla_{i_1} \dots \nabla_{i_k} u 
 \end{equation}
  We then integrate by parts on the right hand side of \( (\ref{3.1}) \) to obtain

 \begin{equation}\label{3.2} \int_M \eta \nabla^{i_k} \nabla^{i_1} \dots \nabla^{i_{k-1}} u \nabla_{i_1} \dots \nabla_{i_k} u =
 - \int_M \nabla^{i_k} \eta \nabla^{i_1} \dots \nabla^{i_{k-1}} u \nabla_{i_1} \dots \nabla_{i_k} u 
  - \int_M \eta \nabla^{i_1} \dots \nabla^{i_{k-1}} u \nabla^{i_k} \nabla_{i_1} \dots \nabla_{i_k} u   \end{equation}
Changing in order of indices on the first term on the right hand side of \( (\ref{3.2}) \) by Lemma A.2 and integrating by parts we obtain

 \begin{align*} - \int_M \nabla^{i_k} \eta \nabla^{i_1} \dots \nabla^{i_{k-1}} &u \nabla_{i_1} \dots \nabla_{i_k} u \equiv - \int_M \nabla^{i_k} \eta \nabla^{i_1} \dots \nabla^{i_{k-1}} u \nabla_{i_k} \nabla_{i_1} \dots \nabla_{i_{k-1}} u  \\
  &= -\int_M \Delta \eta \nabla^{i_1} \dots \nabla^{i_{k-1}} u \nabla_{i_1} \dots \nabla_{i_{k-1}} u + \int_M \nabla^{i_k} \eta \nabla_{i_k} \nabla^{i_1} \dots \nabla^{i_{k-1}} u \nabla_{i_1} \dots \nabla_{i_{k-1}} u    \end{align*}
  and therefore after subtracting the second term over, we obtain 
  \begin{equation}\label{3.3}
  \left|\int_M \nabla^{i_k} \eta \nabla^{i_1} \dots \nabla^{i_{k-1}} u \nabla_{i_1} \dots \nabla_{i_k} u\right| \leq C \sum_{i=0}^{k-1} \int_{supp(\eta)} |\nabla^i u|^2  .\end{equation} 
  Therefore 
  \begin{equation}\label{equiv2} \int_M \eta \nabla^{i_k} \nabla^{i_1} \dots \nabla^{i_{k-1}} u \nabla_{i_1} \dots \nabla_{i_k} u \equiv  - \int_M \eta \nabla^{i_1} \dots \nabla^{i_{k-1}} u \nabla^{i_k} \nabla_{i_1} \dots \nabla_{i_k} u. \end{equation}

 For the right hand side of \( (\ref{equiv2}) \), we once again change the order of the indices with Lemma A.2 and integrate by parts to obtain

 \begin{align} - \int_M  \eta \nabla^{i_1} \dots \nabla^{i_{k-1}}& u \nabla^{i_k} \nabla_{i_1} \dots \nabla_{i_k} u \equiv - \int_M \eta \nabla^{i_1} \dots \nabla^{i_{k-1}} u  \nabla_{i_1} \dots \nabla_{i_{k-1}} \Delta u \notag \\ 
 & \equiv - \int_M \eta \nabla^{i_1} \dots \nabla^{i_{k-1}} u \nabla_{i_{k-1}} \nabla_{i_1} \dots \nabla_{i_{k-2}} \Delta u \notag \\
 &= \int_M \nabla_{i_{k-1}}  \eta \nabla^{i_1} \dots \nabla^{i_{k-1}} u \nabla_{i_1} \dots \nabla_{i_{k-2}} \Delta u + \int_M \eta \nabla_{i_{k-1}} \nabla^{i_1} \dots \nabla^{i_{k-1}} u \nabla_{i_1} \dots \nabla_{i_{k-2}} \Delta u \notag  \\
& \equiv \int_M \nabla_{i_{k-1}} \eta \nabla^{i_1} \dots \nabla^{i_{k-1}} u \nabla_{i_1} \dots \nabla_{i_{k-2}} \Delta u + \int_M \eta  \nabla^{i_1} \dots \nabla^{i_{k-2}}\Delta u \nabla_{i_1} \dots \nabla_{i_{k-2}} \Delta u 
\end{align}

 Then for the first term of \( (4.6) \) we again reorganize the indices to obtain

 \[ \int_M \nabla_{i_{k-1}} \eta \nabla^{i_1} \dots \nabla^{i_{k-1}} u \nabla_{i_1} \dots \nabla_{i_{k-2}} \Delta u \equiv \int_M \nabla_{i_{k-1}} \eta \nabla^{i_1} \dots \nabla^{i_{k-1}} u \Delta \nabla_{i_1} \dots \nabla_{i_{k-2}} u\]
 Integrating by parts gives
 \begin{align*} \int_M \nabla_{i_{k-1}} \eta \nabla^{i_1} \dots \nabla^{i_{k-1}} u \Delta \nabla_{i_1} \dots \nabla_{i_{k-2}} u &= \int_M \nabla^{i_k} \nabla_{i_{k-1}} \eta \nabla^{i_1} \dots \nabla^{i_{k-1}} u  \nabla_{i_{k}}  \nabla_{i_1} \dots \nabla_{i_{k-2}} u  \\
 & + \int_M \nabla_{i_{k-1}} \eta \nabla^{i_k} \nabla^{i_1} \dots \nabla^{i_{k-1}} u \nabla_{i_k}  \nabla_{i_1} \dots \nabla_{i_{k-2}} u \end{align*}
 We clearly have \[  \int_M \nabla^{i_k} \nabla_{i_{k-1}} \eta \nabla^{i_1} \dots \nabla^{i_{k-1}} u  \nabla_{i_{k}}  \nabla_{i_1} \dots \nabla_{i_{k-2}} u  \leq C  \sum_{i=0}^{k-1} \int_{supp(\eta)} |\nabla^i u|^2 \] 
 and by reorganizing indices with Lemma A.2 and then relabeling the indices we obtain 
 \begin{align*} \int_M \nabla_{i_{k-1}} \eta \nabla^{i_k} \nabla^{i_1} \dots \nabla^{i_{k-1}} u \nabla_{i_k}  \nabla_{i_1} \dots \nabla_{i_{k-2}} u &\equiv \int_M \nabla_{i_{k-1}} \eta \nabla^{i_{1}} \dots \nabla^{i_{k-2}} \nabla^{i_{k}} \nabla^{i_{k-1}} u   \nabla_{i_1} \dots \nabla_{i_{k-2}} \nabla_{i_k} u \\
 &= \int_M \nabla^{i_k} \eta \nabla^{i_1} \dots \nabla^{i_{k-1}} u \nabla_{i_1} \dots \nabla_{i_k} u \\
& \leq C  \sum_{i=0}^{k-1} \int_{supp(\eta)} |\nabla^i u|^2 \end{align*} by  \( (\ref{3.3}) \). This shows that
\begin{equation}\label{3.4}
     - \int_M  \eta \nabla^{i_1} \dots \nabla^{i_{k-1}} u \nabla^{i_k} \nabla_{i_1} \dots \nabla_{i_k} u \equiv \int_M \eta \nabla^{i_1} \dots \nabla^{i_{k-2}}\Delta u \nabla_{i_1} \dots \nabla_{i_{k-2}} \Delta u = \int_M \eta |\nabla^{k-2}(\Delta u)|^2
\end{equation}
By the induction hypothesis, we have
\[ \left| \int_M \eta |\nabla^{k-2}(\Delta u)|^2 - \int_M \eta (\Delta^\frac{k}{2} u)^2  \right| \leq C\sum_{i=0}^{k-3} \int_{supp(\eta)} |\nabla^i (\Delta u)|^2 \leq C \sum_{i=0}^{k-1} \int_{supp(\eta)} |\nabla^i u|^2  \]
which implies
 \begin{equation}\label{equiv3}
     \int_M \eta |\nabla^{k-2}(\Delta u)|^2 \equiv \int_M \eta (\Delta^\frac{k}{2} u)^2 .
 \end{equation}
 Therefore by \( (\ref{3.1}),(\ref{equiv2}), (\ref{3.4})\) and \((\ref{equiv3})\) we have proven
Lemma A.3.

\end{proof}

\begin{lem}
    Let \( \eta \) be a compactly supported smooth function, \( k \geq 0 \) be an integer, and \( r > 0 \) be a real number. Then there exists a constant \( C \) based on \( n\), \( k \), bounds for \( R \) and finitely many of its derivatives, \( r \), and bounds for \( \eta \) and finitely many of its derivatives such that for all \( u \in H_k^2(M) \),
   \[ \left| \int_M \eta | \nabla^k u |^2 dv_g - \int_M \eta (\Delta^{\frac{k}{2}} u)^2 \right| \leq C\sum_{i=0}^{k-1} \int_{B_r( supp (\eta))} (\Delta^{\frac{i}{2}} u)^2   \]
\end{lem}

\begin{proof}
 We prove by strong induction. The statement is immediate for \( k = 0 \) and \( k=1\) as the left hand side vanishes. Now fix \( k \geq 2 \) and assume the statement holds for all \( j < k \). Let \( \eta_1 \) be a smooth nonnegative function such that \( \eta_1 = 1 \) on \( supp(\eta) \) and \( \eta_1 = 0 \) on \( M\setminus B_{\frac{r}{2}}(supp(\eta))\).

Then we apply Lemma A.3 and invoke the (strong) induction hypothesis with \( \eta_1 \) taking the place of \( \eta \) and \( \frac{r}{2} \) taking the place of \( r \) to obtain
\begin{align*}
   \left| \int_M \eta | \nabla^k u |^2- \int_M \eta (\Delta^{\frac{k}{2}} u)^2 \right|  &\leq C \sum_{i=0}^{k-1} \int_{supp(\eta)} |\nabla^i u|^2  \\
    &\leq  C  \sum_{i=0}^{k-1} \int_M \eta_1 |\nabla^i u|^2  \\
    & \leq  C  \sum_{i=0}^{k-1} \left( \int_M \eta_1 (\Delta^\frac{i}{2} u)^2 + \sum_{j=0}^{i-1} \int_{B_\frac{r}{2}(supp(\eta_1))} (\Delta^\frac{j}{2} u)^2 \right) \\
    & \leq C\sum_{i=0}^{k-1} \int_{B_r( supp (\eta))} (\Delta^{\frac{i}{2}} u)^2
\end{align*}
    
\end{proof}

\begin{lem}

Let \( k \geq 0 \) be an integer. Let \( \eta \) be a compactly supported smooth function and \( r > 0 \) be a real number. Then there exists a constant \( C \) based on \( n \), \( k \),  bounds for \( R \) and finitely many of its derivatives, \( r \), and bounds for \( \eta \) and finitely many of its derivatives such that 
 \[ \left| \int_M (\Delta^\frac{k}{2} (\eta u))^2 - \int_M \eta^2 (\Delta^\frac{k}{2} u)^2 \right| \leq C\sum_{i=0}^{k-1} \int_{B_r (supp(\eta))} (\Delta^{\frac{i}{2}} u)^2 \]

 \end{lem}
 
 \begin{proof}
    First suppose \( k \) is even. We write \( k=2l\).
 We recall for any tensors \( S_1,S_2 \) of the same type we have
\[ \Delta \langle S_1, S_2 \rangle = \langle S_1 \Delta S_2 \rangle - 2 \langle \nabla S_1, \nabla S_2 \rangle + \langle S_2 \Delta S_1 \rangle \]

Let us define operators  \( T_{i} \) on tensors for \( i \in \{0, 1, 2\} \) by \( T_0 = \text{Id} \), \( T_1 = \nabla \), \( T_2 = \Delta \). Then for any multi-index \( \beta = (\beta_1, \dots, \beta_l) \in \{0, 1, 2\}^l \), we define \( T_\beta = T_{\beta_l} \dots T_{\beta_1} \). Given \( \beta \), we define \( \beta' = (2 - \beta_1, \dots, 2 - \beta_l) \). Finally we define \( j_\beta \) to be the number of \( 1 \)'s appearing in \( \beta \). We note \( j_\beta = j_{\beta'} \), therefore if \( S_1, S_2 \) are tensors of the same type, so are \( T_\beta S_1 \) and \( T_{\beta'}S_2 \). Then we write the \( l \)th Laplacian of the product of two functions as

\begin{equation}\label{5.1.1} \Delta^l (\eta u) = \sum_{\beta \in \{0, 1, 2\}^l } (-2)^{j_\beta} \langle T_\beta u, T_{\beta'} \eta \rangle_g. \end{equation}
This formula can easily be proven by induction.

We then have
\begin{equation}\label{5.1.2} \int_M (\Delta^l (\eta u))^2 = \int_M \sum_{\alpha, \beta \in \{0, 1, 2\}^l } (-2)^{j_\beta + j_\alpha} \langle T_\beta u, T_{\beta'} \eta \rangle_g \langle T_\alpha u, T_{\alpha'} \eta \rangle_g \end{equation}
It follows that to estimate \( \left| \int_M (\Delta^l (\eta u))^2 - \int_M \eta^2 (\Delta^l u)^2 \right|  \), it suffices to estimate the terms on the right hand side of \( (\ref{5.1.2}) \) apart from the principal term of \( \int_M \eta^2 (\Delta^l u)^2 \) corresponding to the case \( |\alpha| = |\beta| = 2l \). We separate this sum into three cases modulo symmetry in \( \alpha \) and \( \beta \).

\textbf{Case 1:} \( |\alpha| \leq 2l-1 \) and \( |\beta| \leq 2l-1 \). In this case, on \( supp(\eta) \), we have \begin{align*} \left| (-2)^{j_\beta + j_\alpha} \langle T_\beta u, T_{\beta'} \eta \rangle_g \langle T_\alpha u, T_{\alpha'} \eta \rangle_g \right| & \leq C |T_\alpha u||T_\beta u| \\
& \leq C(|T_\alpha u|^2 + |T_\beta u|^2) \\ 
& \leq C(|\nabla^{|\alpha|} u|^2 + |\nabla^{|\beta|} u|^2)
\end{align*}
Letting \( \eta_1 \) be a smooth nonnegative function such that \( \eta_1 = 1 \) on \( supp(\eta) \) and \( \eta_1 = 0 \) on \( M\setminus B_{\frac{r}{2}}(supp(\eta))\) and applying Lemma A.4 with \( \frac{r}{2} \) as our value of \( r \) results in
\[ \left| \int_M (-2)^{j_\beta + j_\alpha} \langle T_\beta u, T_{\beta'} \eta \rangle_g \langle T_\alpha u, T_{\alpha'} \eta \rangle_g \right|  \leq C \int_M \eta_1 (|\nabla^{|\alpha|} u|^2 + |\nabla^{|\beta|} u|^2) \leq C\sum_{i=0}^{2l-1} \int_{B_r (supp(\eta))} (\Delta^{\frac{i}{2}} u)^2  \]

\textbf{Case 2:} \( |\alpha| \leq 2l-2 \) and \( |\beta| = 2l \). In this case we integrate by parts and argue as in Case 1 to obtain
\begin{align*} \left| \int_M (-2)^{j_{\alpha}} \eta \Delta^{l} u \langle T_\alpha u, T_{\alpha'} \eta \rangle_g \right| &= \bigg|(-2)^{j_\alpha} \int_M  \langle \nabla \Delta^{l-1} u, \nabla \eta \rangle \langle T_\alpha u, T_{\alpha'} \eta \rangle +   (-2)^{j_{\alpha}} \int_M \eta \langle \nabla^i \Delta^{l-1} u \nabla_i T_\alpha u, T_{\alpha'}\eta \rangle_g \\
& \;\;\;\;\;\;\;\;\;\;\;\;\;\;\;\;\; \quad\quad\quad\quad
+  (-2)^{j_{\alpha}} \int_M \eta \langle \nabla^i \Delta^{l-1} u \nabla_i T_{\alpha'} \eta, T_{\alpha} u \rangle_g \bigg| \\
& \leq C\sum_{i=0}^{2l-1} \int_{B_r (supp(\eta))} (\Delta^{\frac{i}{2}} u)^2
\end{align*}

\textbf{Case 3:} \( |\alpha| = 2l-1 \) and \( |\beta| = 2l \).  In this case we have

\[ \int_M (-2)^{j_\beta + j_\alpha} \langle T_\beta u, T_{\beta'} \eta \rangle_g \langle T_\alpha u, T_{\alpha'} \eta \rangle_g =  -2\int_M \eta \Delta^{l} u \langle \Delta^{\gamma_1}  \nabla \Delta^{\gamma_2} u ,\nabla \eta 
 \rangle \] for some \( \gamma_1 + \gamma_2 = l-1 \). In the following computation (and the rest of the proof), we extend our notation \( A \equiv B \) from the proof of Lemma A.3 to mean \[ |A - B| \leq C  \sum_{i=0}^{2l-1} \int_{supp(\eta)} |\nabla^i u|^2 \leq C  \sum_{i=0}^{2l-1} \int_{B_r(supp(\eta))} |\Delta^\frac{i}{2} u|^2 \] where the last inequality is by an application of Lemma A.4.
Integrating by parts and reorganizing the indices by Lemma A.2 we obtain
\begin{align*}  \int_M \eta \Delta^{l} u \langle \Delta^{\gamma_1} \nabla  \Delta^{\gamma_2} u, \nabla \eta \rangle &= \frac{1}{2} \int_M \Delta^{l} u \langle \Delta^{\gamma_1}  \nabla \Delta^{\gamma_2} u, \nabla (\eta^2) \rangle \\
&\equiv \frac{1}{2} \int_M \Delta^{l} u \langle \nabla \Delta^{l-1} u \nabla (\eta^2) \rangle \\
 &= \frac{1}{2} \int_M \langle \nabla^2\Delta^{l-1} u, \nabla \Delta^{l-1} u \otimes \nabla (\eta^2)\rangle+ \frac{1}{2}\int_M  \langle \nabla^2 (\eta^2),\nabla\Delta^{l-1} u \otimes \nabla\Delta^{l-1} u \rangle  \\
 &= \frac{1}{4} \int_M (\Delta^{\frac{2l-1}{2}} u)^2 \Delta (\eta^2)  + \frac{1}{2}\int_M \langle \nabla \Delta^{l-1} u \otimes \nabla\Delta^{l-1} u, \nabla^2 (\eta^2) \rangle .\\  
  \end{align*}
Arguing as in Case 1 then shows
\[ \left| \int_M (-2)^{j_\beta + j_\alpha} \langle T_\beta u, T_{\beta'} \eta \rangle_g \langle T_\alpha u, T_{\alpha'} \eta \rangle_g  \right| \leq  C\sum_{i=0}^{2l-1} \int_{B_r (supp(\eta))} (\Delta^{\frac{i}{2}} u)^2 \]
and therefore Lemma A.5 holds.

    Now suppose \( k \) is odd. We write \( k = 2l+1 \). Using our computations for \( \Delta^l (\eta u) \), we obtain
\begin{align*} \int_M |\nabla \Delta^l (\eta u)|^2 &= \int_M \sum_{\alpha, \beta \in \{0, 1, 2\}^l } (-2)^{j_\beta + j_\alpha}\left( \langle \nabla^i T_\beta u, T_{\beta'} \eta \rangle_g \langle \nabla_i T_\alpha u, T_{\alpha'} \eta \rangle_g  + \langle T_\beta u, \nabla^i T_{\beta'} \eta \rangle_g \langle \nabla_i T_\alpha u, T_{\alpha'} \eta \rangle_g \right) dv_g \\
& + \int_M \sum_{\alpha, \beta \in \{0, 1, 2\}^l } (-2)^{j_\beta + j_\alpha}\left( \langle \nabla^i T_\beta u, T_{\beta'} \eta \rangle_g \langle  T_\alpha u, \nabla_i T_{\alpha'} \eta \rangle_g  + \langle T_\beta u, \nabla^i T_{\beta'} \eta \rangle_g \langle  T_\alpha u, \nabla_i T_{\alpha'} \eta \rangle \right) dv_g .
\end{align*}

We once again have a principal term of \( \int_M \eta^2 |\nabla \Delta^l u|^2\) as the first of the four terms when \( |\alpha| = |\beta| = 2l. \) Therefore to prove Lemma A.5 we once again split the other terms into cases modulo symmetry and estimate.

\textbf{Case 1:} \( |\alpha| \leq 2l - 1 \) and \( |\beta| \leq 2l-1 \). The details are virtually identical to the proof of Case 1 when \( k \) is even.

\textbf{Case 2:} \( |\alpha| \leq 2l-2 \) and \( \beta = 2l \). In this case, we can still apply the argument from Case 1 to the second term in each row of the above formula. We perform the argument from Case 2 when k is even on the first term in the first row, the first term in the second row is handled similarly. Integrating by parts and estimating as in Case 1 results in
\begin{align*} \left| (-2)^{j_\alpha}\int_M \eta \nabla^i \Delta^l u \langle \nabla_i T_\alpha u, T_{\alpha'} \eta \rangle_g \right| &= \bigg| (-2)^{j_\alpha}\int_M  \Delta^l u \nabla^i \eta \langle \nabla_i T_\alpha u, T_{\alpha'} \eta \rangle_g + (-2)^{j_\alpha} \int_M \eta  \Delta^l u \langle \Delta T_\alpha u, T_{\alpha'} \eta \rangle_g \\
&\;\;\;\;\;\;\;\;\;\;\;\;\;\;\;\;\; \quad\quad\quad\quad + (-2)^{j_\alpha}\int_M \eta \Delta^l u \langle \nabla T_\alpha u, \nabla T_{\alpha'} \eta \rangle_g \bigg| \\
& \leq C\sum_{i=0}^{2l} \int_{B_r (supp(\eta))} (\Delta^{\frac{i}{2}} u)^2
\end{align*}

\textbf{Case 3:} \( |\alpha| = 2l - 1\) and \( |\beta| = 2l \). Once again the second term in each row of the sum can be estimated using the methods from Case 1. The first term in the second row can be estimated using the method from Case 2, leaving us with only one term to consider. We calculate, writing \( T_\alpha = \Delta^{\gamma_1} \nabla \Delta^{\gamma_2} \) where \( \gamma_1 + \gamma_2 = l-1 \) as in Case 3 of (i), and integrating by parts and changing the order of indices with Lemma A.2,
\begin{align*} \int_M \eta \nabla^i \Delta^l u \nabla_i \Delta^{\gamma_1} \nabla^j \Delta^{\gamma_2} u \nabla_j \eta &=  \frac{1}{2} \int_M  \nabla^i \Delta^l u \nabla_i \Delta^{\gamma_1} \nabla^j \Delta^{\gamma_2} u \nabla_j (\eta^2) \\
&\equiv \frac{1}{2}\int_M \nabla^i \Delta^l u \nabla_i \nabla^j \Delta^{l-1} u \nabla_j (\eta^2) \\ 
& = \frac{1}{2}\int_M  \Delta^l u \langle \Delta  \nabla \Delta^{l-1} u \nabla (\eta^2) \rangle - \frac{1}{2}\int_M   \Delta^l u \langle \nabla^2 \Delta^{l-1} u ,\nabla^2 (\eta^2) \rangle \\
& \equiv \frac{1}{2}\int_M  \Delta^l u \langle \nabla  \Delta^{l} u \nabla (\eta^2) \rangle - \frac{1}{2}\int_M \Delta^l u  \langle \nabla^2 \Delta^{l-1} u ,\nabla^2 (\eta^2)\rangle \\
& = \frac{1}{4} \int_M (\Delta^l u)^2 \Delta (\eta^2) - \frac{1}{2}\int_M \Delta^l u \langle \nabla^2 \Delta^{l-1} u, \nabla^2 (\eta^2) \rangle \\
& \leq C\sum_{i=0}^{2l} \int_{B_r (supp(\eta))} (\Delta^{\frac{i}{2}} u)^2
\end{align*}

\textbf{Case 4:} \( |\alpha| = 2l\) and \( |\beta| = 2l \).
The first term is the previously mentioned principal term. The second term in the first row and first term in the second row can be estimated using the method from Case 3. The second term in the second row can be estimated using the method from Case 1. This completes the proof of Lemma A.5.

\end{proof}

\begin{lem}

Let \( \eta \) be a compactly supported smooth function and \( r > 0 \) be a real number. Let \( i_1, i_2, j_1, j_2 \in \frac{1}{2} \mathbb{N} \) be such that \( i_2 + i_2  = j_1 + j_2 = k \). Then there exists a constant \( C \) based on \( n \), \( k \),  bounds for \( R \) and finitely many of its derivatives, \( r \), and bounds for \( \eta \) and finitely many of its derivatives such that 
\[ \left| \int_M \eta \Delta^{i_1} u \Delta^{i_2} u - \int_M \eta \Delta^{j_1} u \Delta^{j_2} u \right| \leq C\sum_{i=0}^{k-1} \int_{B_r (supp(\eta))} (\Delta^{\frac{i}{2}} u)^2 .  \]

\end{lem}

\begin{proof}
We prove \begin{equation}\label{ibp1}\left| \int_M \eta u \Delta^k u - \int_M \eta \Delta u \Delta^{k-1} u \right| \leq C\sum_{i=0}^{k-1} \int_{B_r (supp(\eta))} (\Delta^{\frac{i}{2}} u)^2\end{equation}
and
\begin{equation}\label{ibp2}\left| \int_M \eta u \Delta^k u - \int_M \eta \langle \nabla u, \nabla \Delta^{k-1} u \rangle \right| \leq C\sum_{i=0}^{k-1} \int_{B_r (supp(\eta))} (\Delta^{\frac{i}{2}} u)^2 .\end{equation}
The full statement then follows by a simple induction.

 We calculate

 \begin{equation}\label{5.3.1} \int_M \eta u \Delta^k u = \int_M \Delta(\eta u) \Delta^{k-1}u = \int_M \eta \Delta u \Delta^{k-1} u - 2\int_M \langle \nabla \eta, \nabla u \rangle \Delta^{k-1} u + \int_M u \Delta \eta  \Delta^{k-1} u.   \end{equation}
 We integrate by parts on the final term of \( (\ref{5.3.1}) \) and we have
 \[  \int_M u\Delta \eta \Delta^{k-1} u = \int_M \Delta^\frac{k-1}{2} (u \Delta \eta ) \Delta^\frac{k-1}{2} u \leq C\sum_{i=0}^{k-1} \int_{B_r( supp (\eta))} (\Delta^{\frac{i}{2}} u)^2 \]
 after expanding with \( (\ref{5.1.1}) \) and applying Lemma A.4.
 For the second term of \( (\ref{5.3.1}) \), we integrate by parts and apply \( (\ref{5.1.1}) \) to obtain
 \begin{align*}
     2\int_M \langle \nabla \eta, \nabla u \rangle \Delta^{k-1} u &= 2 \int_M \Delta^{\frac{k-1}{2}} \langle\nabla \eta ,\nabla u \rangle \Delta^{\frac{k-1}{2}} u \\
     &= 2\int_M \langle \nabla \eta ,\Delta^{\frac{k-1}{2}} \nabla u\rangle \Delta^{\frac{k-1}{2}} u + \int_M \sum_{\beta \in \{0, 1, 2\}^k \setminus \{ (2, \dots, 2\} } (-2)^{j_\beta} \langle T_\beta u, T_{\beta'} \eta \rangle_g\Delta^{\frac{k-1}{2}} u 
 \end{align*}
 and computations as in the proof of Case 1 in Lemma A.5 show \[  \int_M \sum_{\beta \in \{0, 1, 2\}^k \setminus \{ (2, \dots, 2\} } (-2)^{j_\beta} \langle T_\beta u, T_{\beta'} \eta \rangle_g\Delta^{\frac{k-1}{2}} u  \leq  C\sum_{i=0}^{k-1} \int_{B_r( supp (\eta))} (\Delta^{\frac{i}{2}} u)^2.\]
For the remaining term,  we reorganize the indices with Lemma A.2 to obtain
 \[ \int_M \langle \nabla \eta ,\Delta^{\frac{k-1}{2}} \nabla u \rangle \Delta^{\frac{k-1}{2}} u \equiv \int_M \langle \nabla \eta, \nabla \Delta^{\frac{k-1}{2}}  u \rangle \Delta^{\frac{k-1}{2}} u = \frac{1}{2} \int_M \Delta \eta (\Delta^{\frac{k-1}{2}} u )^2 \leq C \int_{supp(\eta)} (\Delta^\frac{k-1}{2} u)^2  \]
 and we have shown \( (\ref{ibp1}) \).
    
    For \( (\ref{ibp2} \), we integrate by parts twice to obtain
    \begin{align*} \int_M \eta u \Delta^k u &= \int_M \langle \nabla (\eta u), \nabla \Delta^{k-1} u \rangle \\
    &= \int_M \eta \langle  \nabla u, \nabla \Delta^{k-1} u\rangle + \int_M u\langle \nabla \eta ,\nabla \Delta^{k-1} u \rangle \\
    &=  \int_M \eta\langle \nabla u ,\nabla \Delta^{k-1} u\rangle + \int_M \langle \nabla u ,\nabla \eta \rangle \Delta^{k-1} u +  \int_M u \Delta \eta \Delta^{k-1} u . 
    \end{align*}
    The latter two terms are identical to the latter two terms of \( (\ref{5.3.1}) \) and we conclude by applying the same computations.
\end{proof}

 \begin{lem}
   Let \( k \geq 0 \) be an integer and \( i_1, j_1, i_2, j_2\) be integers such such that \( i_1 + i_2 = j_1 + j_2 = k \). Let \( r > 0 \) be a real number. Let \( \eta \) be a compactly supported smooth function. Then there exists \( C \) based on \( n, k, r \), bounds for \( R \) and finitely many of its derivatives, and bounds for \( \eta \) and finitely many of its derivatives such that for all \( u \in H_k^2(M) \) and \( \beta \geq 0 \),

     \[  \int_M \eta (\Delta + \beta)^{i_1}u (\Delta + \beta)^{i_2} u \leq \int_M \eta (\Delta + \beta)^{j_1}u (\Delta + \beta)^{j_2} u + C \int_{B_r (supp(\eta))} \sum_{i=0}^{k-1}\sum_{j=0}^{k-1-i} \beta^{k-1-i-j} (\Delta^\frac{j}{2} u)^2   \]
     where \( C \) is independent of \( u \) and \( \beta \).
\end{lem}

\begin{proof}
We begin by stating a simple but useful reindexing identity. Let \( a_i, b_i \) be real numbers for \( 0 \leq i \leq k \). Then
\begin{equation}\label{5.r} \sum_{i=0}^k \sum_{j=0}^{k-i} a_{k-i-j} b_j = \sum_{i=0}^k a_{k-i} \sum_{j=0}^i b_j  \end{equation}
To prove this we use the substitution \( i'=i-j\) to obtain
\[  \sum_{i=0}^k a_{k-i} \sum_{j=0}^i b_j = \sum_{i=0}^k \sum_{j=0}^k a_{k-i} b_j \mathbbm{1}_{j\leq i} = \sum_{j=0}^k \sum_{i=0}^k a_{k-i} b_j \mathbbm{1}_{j\leq i}= \sum_{j=0}^k \sum_{i'=0}^{k-j} a_{k-i'-j} b_j =  \sum_{i'=0}^k \sum_{j=0}^{k-i'} a_{k-i'-j} b_j \]

 Now to prove Lemma A.7, we will show for all \( k_1, k_2 \) such that \( k_1 + k_2 = k \) and \( k_1 < k_2 \)
    \begin{equation}\label{5.1} \left|\int_M \eta (\Delta + \beta)^{k_1} u (\Delta + \beta)^{k_2} -  \int_M \eta (\Delta + \beta)^{k_1+1} u (\Delta + \beta)^{k_2-1}\right| \leq C \int_{B_r(supp(\eta))} \sum_{i=0}^{k-1}\sum_{j=0}^{k-1-i} \beta^{k-1-i-j} (\Delta^\frac{j}{2} u)^2 \end{equation}
    and full statement immediately follows by a simple induction. First we prove the case of \( (\ref{5.1})\) where \( k_1 = 0, k_2 = k \). We expand, defining \( c_{i, k} \) such that \( (\Delta+\beta)^k = \sum_{i=0}^kc_{i, k}\beta^{k-i} \Delta^i u \) to get
    \begin{align*}
    \int_M \eta u (\Delta + \beta)^k u & = \int_M \eta u \Delta((\Delta + \beta)^{k-1} u) + \int_M \beta \eta u (\Delta + \beta)^{k-1} u)  \\
    &= \int_M \eta u \Delta\left(\sum_{i=0}^{k-1} c_{i,k-1} \beta^{k-1-i}\Delta^i u \right) + \int_M \beta \eta u (\Delta + \beta)^{k-1} u \\
    &= \int_M \sum_{i=0}^{k-1} c_{i,k-1}  \beta^{k-1-i} \eta u  \Delta^{i+1} u  + \int_M \beta \eta u (\Delta + \beta)^{k-1} u \\
    \end{align*}
    We then apply Lemma A.6 to each term and reindex with \( (\ref{5.r}) \) to obtain
    \begin{align*}
    \int_M \sum_{i=0}^{k-1} c_{i,k-1}  \beta^{k-1-i} \eta u  \Delta^{i+1} u  + \int_M \beta \eta u (\Delta + \beta)^{k-1} u &\leq  \int_M \eta \Delta u \sum_{i=0}^{k-1} c_{i,k-1} \beta^{k-1-i} \Delta^i u +  \beta \eta u (\Delta + \beta)^{k-1} u  \\
    & \quad\quad\quad\quad\quad + C\int_{B_r(supp(\eta))} \sum_{i=0}^{k-1} \beta^{k-1-i} \sum_{j=0}^i (\Delta^\frac{j}{2})^2 u \\
    &=\int_M \eta (\Delta + \beta) u (\Delta + \beta)^{k-1} u \\
    & \quad\quad\quad\quad+ C \int_{B_r(supp(\eta))} \sum_{i=0}^{k-1}\sum_{j=0}^{k-1-i} \beta^{k-1-i-j} (\Delta^\frac{j}{2} u)^2
    \end{align*}

    Now let \( k_1, k_2 \) be such that \( k_1 + k_2 = k \) and assume \( k_1 < k_2 \). 
    We write 
    \[ \int_M \eta (\Delta + \beta)^{k_1} u (\Delta + \beta)^{k_2}u = \int_M \eta (\Delta + \beta)^{k_1} u (\Delta + \beta)^{k_2-k_1} (\Delta + \beta)^{k_1} u \]
    and apply the previous special case replacing \( u \) with \( (\Delta + \beta)^{k_1} u\) and \( r \) with \( \frac{r}{3} \). We obtain
    \begin{align} \int_M \eta (\Delta + \beta)^{k_1} u (\Delta + \beta)^{k_2}u &\leq \int_M \eta (\Delta + \beta)^{k_1+1} u (\Delta + \beta)^{k_2-1} \notag \\
    & \quad \quad\quad\quad\quad+ C \sum_{i=0}^{k_2-k_1-1}\sum_{j=0}^{k_2-k_1-1-i} \beta^{k_2-k_1-1-i-j} \int_{B_\frac{r}{3}(supp(\eta))} (\Delta^\frac{j}{2} (\Delta + \beta)^{k_1} u )^2 \label{1.5.end}  \end{align}
Let \( \eta_1 \) be a smooth function equal to \( 1 \) on \( B_\frac{r}{3}(supp(\eta)) \) and equal to \( 0 \) on \( M \setminus B_\frac{2r}{3}(supp(\eta)) \) We fix an arbitrary \( i \) and \( j \) and consider \( \beta^{k_2-k_1-1-i-j} \int_M \eta_1 (\Delta^\frac{j}{2} (\Delta + \beta)^{k_1} u )^2\). We expand
\begin{align} \int_M \eta_1 (\Delta^\frac{j}{2} (\Delta + \beta)^{k_1} u )^2 &= \int_M \eta_1 \left(\sum_{l=0}^{k_1} c_{l,k_1} \beta^{k_1-l} \Delta^{l + \frac{j}{2}}u \right)^2 \notag \\
&=  \sum_{0 \leq l_1, l_2 \leq k_1} \int_M \eta_1 c_{l_1, k_1}c_{l_2, k_1} \beta^{2k_1-l_1-l_2}\Delta^{l_1+\frac{j}{2}}u\Delta^{l_2+\frac{j}{2}}u \label{lemma7eq1} \end{align}
We rewrite this sum as
\begin{equation}\label{6.2}   \sum_{0 \leq l_1, l_2 \leq k_1} \int_M \eta_1 c_{l_1, k_1}c_{l_2, k_1} \beta^{2k_1-l_1-l_2}\Delta^{l_1+\frac{j}{2}}u\Delta^{l_2+\frac{j}{2}}u =   \sum_{l=0}^{2k_1} \sum_{\substack{0 \leq l_1, l_2 \leq k_1 \\ l_1+l_2=l}} \int_M \eta_1 c_{l_1, k_1}c_{l_2, k_1} \beta^{2k_1-l}\Delta^{l_1+\frac{j}{2}}u\Delta^{l_2+\frac{j}{2}}u  \end{equation} 

Fixing some \( l, l_1, l_2 \) such that \( l_1 + l_2 = l \) and applying Lemma A.6 with \( \eta_1 \) replacing \( \eta \), and \( \frac{r}{3} \)  replacing \( r \) we have
\begin{align*}
   \int_M \eta_1 c_{l_1, k_1}c_{l_2, k_1} \beta^{2k_1-l}\Delta^{l_1+\frac{j}{2}}u   \Delta^{l_2+\frac{j}{2}}u    &\leq  c_{l_1, k_1}c_{l_2, k_1} \beta^{2k_1-l}\left(\int_M \eta_1 \Delta^{\frac{j+l}{2}}u\Delta^{\frac{j+l}{2}}u + C\sum_{p=0}^{l+j-1} \int_{B_r(supp(\eta))} (\Delta^\frac{p}{2} u)^2 \right) \\
&\leq C \beta^{2k_1-l} \sum_{p=0}^{l+j} \int_{B_r(supp(\eta))} (\Delta^\frac{p}{2} u)^2  \\
\end{align*}

Applying this inequality in the right hand side of \( (\ref{6.2}) \) and using the fact that \( k_2 - k_1 + 2k_1 = k \) we obtain
\begin{equation}\label{lemma7eq2} \beta^{k_2-k_1-1-i-j} \sum_{l=0}^{2k_1} \sum_{\substack{0 \leq l_1, l_2 \leq k_1\\ l_1+l_2=l}} \int_M \eta_1 c_{l_1, k_1}c_{l_2, k_1} \beta^{2k_1-l}\Delta^{l_1+\frac{j}{2}}u\Delta^{l_2+\frac{j}{2}}u   \leq C  \sum_{l=0}^{2k_1} \beta^{k-1-i-j-l}\sum_{p=0}^{l+j} \int_{B_r(supp(\eta))} (\Delta^\frac{p}{2} u)^2  \end{equation}

Then substituting \( l+j \) with \(l' \), using the fact that \( 2k_1 + j \leq k - 1 - i \), and applying \( (\ref{5.r}) \) we obtain
\begin{align} \sum_{l=0}^{2k_1} \beta^{k-1-i-j-l}\sum_{p=0}^{l+j} \int_{B_r(supp(\eta))} (\Delta^\frac{p}{2} u)^2 &= \sum_{l'=j}^{2k_1+j} \beta^{k-1-i-l'}\sum_{p=0}^{l'}\int_{B_r(supp(\eta))} (\Delta^\frac{p}{2} u)^2 \notag \\
 &\leq   \sum_{l'=0}^{k-1-i} \beta^{k-1-i-l'}\sum_{p=0}^{l'} \int_{B_r(supp(\eta))}(\Delta^\frac{p}{2} u)^2 \notag \\ 
&= \sum_{l'=0}^{k-1-i}\sum_{p=0}^{k-1-i-l'} \beta^{k-1-i-l'-p} \int_{B_r(supp(\eta))}(\Delta^\frac{p}{2}u)^2. \label{lemma7eq3}
\end{align}
Then by substituting \( l'+i \) with \( l'' \) we obtain
\begin{align} \sum_{l'=0}^{k-1-i}\sum_{p=0}^{k-1-i-l'} \beta^{k-1-i-l'-p} \int_{B_r(supp(\eta))}(\Delta^\frac{p}{2}u)^2 &= \sum_{l''=i}^{k-1}\sum_{p=0}^{k-1-l''} \beta^{k-1-l''-p} \int_{B_r(supp(\eta))}(\Delta^\frac{p}{2}u)^2 \notag \\
&\leq \sum_{l''=0}^{k-1}\sum_{p=0}^{k-1-l''} \beta^{k-1-l''-p} \int_{B_r(supp(\eta))}(\Delta^\frac{p}{2}u)^2. \label{lemma7eq4} \end{align}
Therefore, putting together \( (\ref{lemma7eq1}), (\ref{lemma7eq2}), (\ref{lemma7eq3}), (\ref{lemma7eq4})\) we obtain \[\beta^{k_2-k_1-1-i-j} \int_M \eta_1 (\Delta^\frac{j}{2} (\Delta + \beta)^{k_1} u )^2 \leq \sum_{l''=0}^{k-1}\sum_{p=0}^{k-1-l''} \beta^{k-1-l''-p} \int_{B_r(supp(\eta))}(\Delta^\frac{p}{2}u)^2 \]
Becuase \( i, j \) were arbitrary, we apply this bound to each term of \( (\ref{1.5.end}) \) to conclude
\[ \sum_{i=0}^{k_2-k_1-1}\sum_{j=0}^{k_2-k_1-1-i} \beta^{k_2-k_1-1-i-j} \int_{B_\frac{r}{3}(supp(\eta))} (\Delta^\frac{j}{2} (\Delta + \beta)^{k_1} u )^2 \leq C \int_{B_r(supp(\eta))}  \sum_{i=0}^{k-1}\sum_{j=0}^{k-1-i} \beta^{k-1-i-j} (\Delta^\frac{j}{2} u)^2 \]
and we are done.
\end{proof}

\begin{lem}
   Let \( k \geq 1 \) and \( r > 0 \) be arbitrary and let \( \eta \) be a smooth compactly supported function. Let \( i_1, i_2 \) and \( j_1, j_2 \) be such that \( i_1 + i_2 = j_1 + j_2 = k \). Additionally, for each \( 0 \leq l \leq k-1 \) let \( p_l, q_l \) be an arbitrary pair of nonnegative integers satisfying \( p_l + q_l = l \). Then there exists \( C \) based on \( n, k, r \), bounds for \( R \) and finitely many of its derivatives, and bounds for \( \eta \) and finitely many of its derivatives such that for all \( u \in H_k^2(M) \) and \( \beta \geq 0 \) such that \( (\Delta + \beta)^l u \geq 0 \) for all \( 0 \leq l \leq k \),

    \[  \int_M \eta (\Delta + \beta)^{i_1}u (\Delta + \beta)^{i_2} u \leq \int_M \eta (\Delta + \beta)^{j_1}u (\Delta + \beta)^{j_2} u + C \sum_{l=0}^{k-1} \int_{B_r(supp(\eta))} (\Delta + \beta)^{p_l} u (\Delta + \beta)^{q_l} u \]
\end{lem}

\begin{proof}
Let \( \eta \in C_c^\infty(M) \), \( k \geq 0 \), \( p_l, q_l \) such that \( p_l + q_l = l \) for \( 0 \leq l \leq k \), \( r > 0 \) be arbitrary. Let \( \eta_1 \) be a function such that \( \eta_1 = 1 \) on \( B_\frac{r}{3}(supp(\eta)) \) and \( \eta_1 = 0 \) on \( M \setminus B_{\frac{2}{3}r} (supp(\eta)) \). By applying Lemma A.7 with \( \frac{r}{3} \) replacing \( r \) we obtain
\begin{align*} \left| \int_M \eta (\Delta + \beta)^{i_1}u (\Delta + \beta)^{i_2} u - \int_M \eta (\Delta + \beta)^{j_1}u (\Delta + \beta)^{j_2} u \right| &\leq C \int_{B_\frac{r}{3} (supp(\eta))} \sum_{i=0}^{k-1}\sum_{j=0}^{k-1-i} \beta^{k-1-i-j} (\Delta^\frac{j}{2} u)^2  \\
& \leq C\int_M \sum_{i=0}^{k-1}\sum_{j=0}^{k-1-i} \eta_1 \beta^{k-1-i-j} (\Delta^\frac{j}{2} u)^2 
\end{align*}
It therefore suffices to prove for \( k \geq 0 \), \( \eta_1 \in C_c^\infty(M) \), \( s > 0 \) there exists \( C \) such that
\[ \int_M \sum_{i=0}^{k}\sum_{j=0}^{k-i} \eta_1 \beta^{k-i-j} (\Delta^\frac{j}{2} u)^2 \leq C \sum_{l=0}^{k} \int_{B_s(supp(\eta_1))} (\Delta + \beta)^{p_l} u (\Delta + \beta)^{q_l} u  \]
and letting \( s = \frac{r}{3} \) will allow us to conclude.

 The base case \( k=0 \) is immediate. Assume the statement is true for some \( k \) and let \( p_l, q_l \) satisfy \( p_l + q_l = l \) for \( 1 \leq l \leq k+1 \). Then we have after substituting \( i' = i-1\)
 \begin{align} \int_M \sum_{i=0}^{k+1}\sum_{j=0}^{k+1-i} \eta_1 \beta^{k+1-i-j} (\Delta^\frac{j}{2} u)^2 &= \int_M \sum_{j=0}^{k+1} \eta_1 \beta^{k+1-j} (\Delta^\frac{j}{2}u)^2 + \sum_{i=1}^{k+1}\sum_{j=0}^{k+1-i} \eta_1 \beta^{k+1-i-j} (\Delta^\frac{j}{2} u)^2 \notag \\
 &= \int_M \sum_{j=0}^{k+1} \eta_1 \beta^{k+1-j} (\Delta^\frac{j}{2}u)^2 + \sum_{i'=0}^{k}\sum_{j=0}^{k-i'} \eta_1 \beta^{k-i'-j} (\Delta^\frac{j}{2} u)^2 \label{lemma8} 
 \end{align}
By the induction hypothesis  \[ \sum_{i'=0}^{k}\sum_{j=0}^{k-i'} \int_M \eta_1 \beta^{k-i'-j} (\Delta^\frac{j}{2} u)^2 \leq C \sum_{l=0}^{k} \int_{B_s (supp(\eta_1))} (\Delta + \beta)^{p_l} u (\Delta + \beta)^{q_l} u .\]
 Now to bound the the first term in \( (\ref{lemma8})\), using the positivity assumption on \( (\Delta + \beta)^l u \) we expand
 \begin{align*} 
 \int_{B_\frac{s}{3}(supp(\eta_1))} (\Delta + \beta)^{p_{k+1}} u (\Delta + \beta)^{q_{k+1}} u &\geq \int_M \eta_1 (\Delta + \beta)^{p_{k+1}} u (\Delta + \beta)^{q_{k+1}} u \\
 &= \int_M \eta_1 \left( \sum_{l=0}^{p_{k+1}} c_{l,p_{k+1}} \beta^{p_{k+1}-l}  \Delta^l u\right)\left( \sum_{l=0}^{q_{k+1}} c_{l,q_{k+1}} \beta^{q_{k+1}-l} \Delta^l u \right) \\
 &= \sum_{l=0}^{k+1} \sum_{\substack{0\leq l_1 \leq p_{k+1} \\ \,0\leq l_2 \leq q_{k+1} \\  l_1 + l_2 = l}} \int_M \eta_1 c_{l_1,p_{k+1}}c_{l_2,q_{k+1}}  \beta^{k+1-l}  \Delta^{l_1} u   \Delta^{l_2} u dv_g
 \end{align*}
 Then applying Lemma A.6 to each term we obtain
\[ \sum_{l=0}^{k+1}  \sum_{\substack{0\leq l_1 \leq p_{k+1} \\ \,0\leq l_2 \leq q_{k+1} \\  l_1 + l_2 = l}} \int_M \eta_1 c_{l_1,p_{k+1}}c_{l_2,q_{k+1}}  \beta^{k+1-l}  \Delta^{l_1} u   \Delta^{l_2} u  \geq \sum_{l=0}^{k+1} d_{l,k+1}   \beta^{k+1-l} \left(  \int_M \eta_1(\Delta^\frac{l}{2} u)^2 - C\sum_{p=0}^{l-1} \int_{B_\frac{s}{3}(supp(\eta_1))} (\Delta^\frac{p}{2} u)^2\right) \]
where \[ d_{l,k+1} =  \sum_{\substack{0\leq l_1 \leq p_{k+1} \\ 0\leq l_2 \leq q_{k+1} \\  l_1 + l_2 = l}} c_{l_1,p_{k+1}}c_{l_2, q_{k+1}} \geq 1. \] 
We therefore obtain
\[ \int_{B_\frac{s}{3}(supp(\eta_1))} (\Delta + \beta)^{p_{k+1}} u (\Delta + \beta)^{q_{k+1}}  \geq \sum_{l=0}^{k+1} \beta^{k+1-l} \left(  \int_M \eta_1(\Delta^\frac{l}{2} u)^2 - C\sum_{p=0}^{l-1} \int_{B_\frac{s}{3}(supp(\eta_1))} (\Delta^\frac{p}{2} u)^2\right) . \]
Adding the remainder terms onto the left hand side and substituting \( l' = l-1\) results in
\[ \sum_{l=0}^{k+1}\beta^{k+1-l}\int_M \eta_1 (\Delta^\frac{l}{2} u)^2 \leq \int_{B_\frac{s}{3}(supp(\eta_1))} (\Delta + \beta)^{p_{k+1}} u (\Delta + \beta)^{q_{k+1}} u + C \sum_{l'=0}^{k}   \beta^{k-l'} \sum_{p=0}^{l'} \int_{B_\frac{s}{3}(supp(\eta_1))} (\Delta^\frac{p}{2} u)^2 \]

Let \( \eta_2 \) be such that \( \eta_2 = 1 \) on \( B_{\frac{s}{3}}(supp(\eta_1)) \) and \( \eta_2 = 0 \) on \( M \setminus B_{\frac{2}{3}s} (supp(\eta_1)) \). By \( (\ref{5.r}) \) and the induction hypothesis (applied to \( \eta_2 \) and \( \frac{s}{3} \)) we continue

\begin{align*}
    &\int_{B_\frac{s}{3}(supp(\eta_1))} (\Delta +\beta)^{p_{k+1}} u (\Delta + \beta)^{q_{k+1}} u + C \sum_{l'=0}^{k}   \beta^{k-l'} \sum_{p=0}^{l'} \int_M \eta_2 (\Delta^\frac{p}{2} u)^2 \\
    = &\int_{B_\frac{s}{3}(supp(\eta_1))} (\Delta + \beta)^{p_{k+1}} u (\Delta + \beta)^{q_{k+1}} u + C \sum_{l'=0}^{k}  \sum_{p=0}^{k-l'}   \int_M \beta^{k-l'-p}  \eta_2 (\Delta^\frac{p}{2}u)^2 \\
    \leq &\int_{B_s(supp(\eta_1))} (\Delta + \beta)^{p_{k+1}} u (\Delta + \beta)^{q_{k+1}} u +C \sum_{l=0}^{k} \int_{B_s(supp(\eta_1))} (\Delta + \beta)^{p_l} u (\Delta + \beta)^{q_l} u
\end{align*}
and the claim is proven.

\end{proof}

\begin{lem}

Let \( k \geq 0 \) be an integer and \( u \in H_k^2(M) \) be supported inside a geodesic ball \( B \subset M \) centered around \( x_0 \in M \). Let \( \xi \) represent the Euclidean metric on \( B \) defined through a geodesic normal coordinate chart and \( dx \) the corresponding Euclidean volume element.  Let \( r_g \) be the geodesic distance function to \( x_0 \). Then there exists \( C \) independent of \( u \) such that
    \[  \int_M |\nabla^k_\xi u|^2 dx \leq \int_M (1 + C r_g^2)\left(|\nabla^k_g u|^2 + C \sum_{i=0}^{k-1}|\nabla^i_g u|^2 \right) dv_g \]
\end{lem}

\begin{proof}
    The estimate \( dx \leq (1 + Cr_g^2) dv_g \) implies it is sufficient to prove \[  \int_M |\nabla^k_\xi u|^2 dx \leq \int_M \left( (1 + C r_g^2)|\nabla^k_g u|^2 + C \sum_{i=0}^{k-1}|\nabla^i_g u|^2  \right) dx \]

    We prove by strong induction on \( k \). The base case \( k=0 \) is immediate. Assume the statement holds for all values strictly smaller than same \( k \geq 1 \). We recall the formula for the components of the covariant derivative of a \( p \) covariant tensor \( T \) is 
\[ \nabla_i T_{i_1 \dots i_p} = \frac{\partial T_{i_1 \dots i_p}}{\partial x_i} - \sum_{k=1}^p \Gamma_{i i_k}^\alpha T_{i_1 \dots i_{k-1} \alpha i_{k+1} \dots i_p}. \]
Therefore, applying this to the \( k \)th covariant derivative we obtain
\begin{equation}\label{8.1} \partial_{i_1 \dots i_k}u = {\nabla}^{(g)}_{i_1 \dots i_k}u + \sum_{\alpha \in \{1, \dots, n\}^{k-1}} f_\alpha \partial_{\alpha} u + \sum_{j=0}^{k-2} \sum_{\beta \in \{1, \dots, n\}^{j} }g_\beta \partial_\beta u \end{equation}
where the left hand side represents the Euclidean partial derivative in a geodesic chart, \(  {\nabla}^{(g)}_{i_1 \dots i_k}u \) represents the \( i_1 \dots i_k \) component of the tensor \( {\nabla^k} u \) in the same chart, \( f_\alpha \) represents a sum of positive and negative Christoffel symbols and \( g_\beta \) is a sum of positive and negative derivatives of Christoffel symbols. We have \( |f_\alpha| \leq Cr_g \) and \( |g_\alpha| \leq C \) in our chart. We take the equation
\[ \int_M (\partial_{i_1 \dots i_k}u)^2 dx = \int_M \left(\nabla^{(g)}_{i_1 \dots i_k}u + \sum_{\alpha \in \{1, \dots, n\}^{k-1}} f_\alpha \partial_{\alpha} u + \sum_{j=0}^{k-2} \sum_{\beta \in \{1, \dots, n\}^{j} }g_\beta \partial_\beta u \right)^2 dx \] and consider the resulting terms from the right hand side.

Terms of the form \( \int_M f_{\alpha_1} f_{\alpha_2} \partial_{\alpha_1} u \partial_{\alpha_2} u dx \) where \( |\alpha_1| = |\alpha_2| = k-1 \) can be bounded by taking
\begin{align*} \int_M f_{\alpha_1} f_{\alpha_2} \partial_{\alpha_1} u \partial_{\alpha_2} u dx &\leq C \int_M \partial_{\alpha_1} u \partial_{\alpha_2} u dx \\
& \leq C \int_M |\partial_{\alpha_1} u|^2 + |\partial_{\alpha_2} u|^2 dx \\
& \leq C \int_M |\nabla^{k-1}_\xi u|^2 dx \\
& \leq C \int_M \sum_{i=0}^{k-1}|\nabla^i_g U_\alpha|^2  dx
\end{align*}
where the final inequality is by the induction hypothesis (although the full strength of the estimate is not used).
Terms of the form \( \int_M g_{\beta_1} g_{\beta_2} \partial_{\beta_1} u \partial_{\beta_2} u dx \) and \( \int_M f_\alpha g_\beta \partial_\alpha u \partial_\beta u \) can be bounded through the same argument.

For terms of the form \( \int_M g_\beta \nabla^{(g)}_{i_1 \dots i_k}u \partial_\beta u  \), we first substitute
\[ \nabla^{(g)}_{i_1 \dots i_k}u = \partial_{i_1 \dots i_k}u- \sum_{\alpha \in \{1, \dots, n\}^{k-1}} f_\alpha \partial_{\alpha} u - \sum_{j=0}^{k-2} \sum_{\beta \in \{1, \dots, n\}^{j} } g_\beta \partial_\beta u \]
After making this substitution and expanding, we only need to bound the term \( \int_M g_\beta \partial_{i_1 \dots i_k}u \partial_\beta u dx \), the others can be bounded by the above computations. Euclidean integration by parts gives
\begin{align*}  \int_M g_\beta \partial_{i_1 \dots i_k}u \partial_\beta u dx &= -\int_M \partial_{i_1} g_\beta \partial_{i_2 \dots i_k}u \partial_\beta u dx - \int_M g_\beta  \partial_{i_2 \dots i_k}u \partial_{i_1} \partial_\beta u dx \\
&\leq C \int_M |\nabla^{k-1}_\xi u|^2 dx + C\int_M |\nabla^{|\beta|}_\xi u|^2  dx + C\int_M |\nabla^{|\beta|+1}_\xi u|^2  dx.
\end{align*}
Because \( |\beta| + 1 \leq k-1 \), by the induction hypothesis \[ \int_M |\nabla^{k-1}_\xi u|^2 dx + \int_M |\nabla^{|\beta|}_\xi u|^2  dx + \int_M |\nabla^{|\beta|+1}_\xi u|^2  dx \leq C \int_M \sum_{i=0}^{k-1}|\nabla^i_g U_\alpha|^2  dx \]

For terms of the form \( \int_M f_\alpha \nabla^{(g)}_{i_1 \dots i_k}u \partial_\alpha u dx \), we calculate
\begin{align*}
    \int_M f_\alpha \nabla^{(g)}_{i_1 \dots i_k}u \partial_\alpha u &\leq C \int_M r_g \nabla^{(g)}_{i_1 \dots i_k}u \partial_\alpha u \\
    &\leq C \int_M r_g^2 (\nabla^{(g)}_{i_1 \dots i_k}u)^2 dx + C \int_M ( \partial_\alpha u)^2 dx
    \end{align*}
We have \[ \int_M ( \partial_\alpha u)^2 dx \leq C \int_M \sum_{i=0}^{k-1}|\nabla^i_g u|^2  dx \] by the induction hypothesis. Combining all of our previous computations we obtain
\[ \int_M (\partial_{i_1 \dots i_k}u)^2 dx \leq \int_M (1 + Cr_g^2)(\nabla^{(g)}_{i_1 \dots i_k}u)^2 + C \int_M \sum_{i=0}^{k-1}|\nabla^i_g u|^2 \]
To complete the proof of the lemma, we first recall \( \delta^{ij} \leq (1 + r_g^2) g^{ij} \) where \( \delta^{ij} \) represents the Kronecker delta and calculate
\begin{align*} \int_M |\nabla^k_\xi U_\alpha|^2 dx &\leq \int_M \sum_{1\leq i_1,\dots, i_k\leq n} (1 + Cr_g^2)(\nabla^{(g)}_{i_1 \dots i_k}u)^2 dx + C \int_M \sum_{i=0}^{k-1}|\nabla^i_g u|^2 \\
&=  \int_M (1 + Cr_g^2)\delta^{i_1 j_1} \dots \delta^{i_k j_k} \nabla^{(g)}_{j_1 \dots j_k}u \nabla^{(g)}_{i_1 \dots i_k}u dx + C \int_M \sum_{i=0}^{k-1}|\nabla^i_g u|^2 \\
& \leq \int_M (1 + C r_g^2)|\nabla^k_g U_\alpha|^2 + C \sum_{i=0}^{k-1}|\nabla^i_g U_\alpha|^2  dx
\end{align*}

\end{proof}

\section*{Acknowledgements}
The author would like to thank his supervisor, Jérôme Vétois, for helpful comments and discussions throughout the construction of this work.

\printbibliography

\end{document}